\documentclass{article}
\usepackage[utf8]{inputenc}
\usepackage{enumerate}
\usepackage{amsfonts,amsmath,amssymb,amsthm,bbm}
\usepackage{mathtools}
\usepackage{xcolor}
\usepackage{enumitem}
\usepackage{float}

\usepackage{utfsym}
\usepackage{graphicx}
\usepackage{caption}
\usepackage{subcaption}
\usepackage[colorlinks]{hyperref}

\usepackage{booktabs}
\usepackage{longtable}

\usepackage[style=alphabetic,backend=biber, sorting = nyt, maxnames=99]{biblatex}
\def\calB{{\mathcal{B}}}

\def\calD{{\mathcal{D}}}
\def\calE{{\mathcal{E}}}
\def\calF{{\mathcal{F}}}
\def\calG{{\mathcal{G}}}

\def\calN{{\mathcal{N}}}
\def\calO{{\mathcal{O}}}
\def\calP{{\mathcal{P}}}

\def\sP{{\mathbb{P}}}

\def\sR{{\mathbb{R}}}

\newcommand{\WD}{\mathcal{W}}
\newcommand{\AWD}{\mathcal{AW}}
\newcommand{\AWDsigma}{\mathcal{AW}^{(\sigma)}}

\newcommand{\TVD}{\mathrm{TV}}
\newcommand{\AVD}{\mathrm{AV}}

\newcommand{\cpl}{\text{Cpl}}
\newcommand{\bccpl}{\text{Cpl}_{\text{bc}}}

\newcommand{\E}{\mathbb{E}}
\newcommand{\R}{\mathbb{R}}
\newcommand{\N}{\mathbb{N}}

\newcommand{\Var}{\mathrm{Var}}

\newcommand{\defeq}{\vcentcolon=}

\usepackage{pifont}
\newcommand{\cmark}{\ding{51}}%
\newcommand{\xmark}{\ding{55}}%

\usepackage{thmtools}
\usepackage{hyperref}
\usepackage[capitalise]{cleveref}
\newtheorem{theorem}{Theorem}[section]

\newtheorem{lemma}[theorem]{Lemma}

\theoremstyle{remark}
\newtheorem{remark}[theorem]{Remark}

\theoremstyle{definition}
\newtheorem{definition}[theorem]{Definition}
\newtheorem{example}[theorem]{Example}

\title{Convergence of Adapted Smoothed Empirical Measures}

\author{Songyan Hou\thanks{Department of Mathematics, ETH Z\"{u}rich, Switzerland.~~~\emph{songyan.hou@math.ethz.ch}\newline The author would like to thank Beatrice Acciaio for insightful discussions and valuable advice.} 
 }
\date{\today }

\begin{document}

\maketitle

\begin{abstract}
The \textit{adapted Wasserstein distance} ($\mathcal{AW}$-distance) controls the calibration errors of optimal values in various stochastic optimization problems, pricing and hedging problems, optimal stopping problems, etc. However, statistical aspects of the $\mathcal{AW}$-distance are bottlenecked by the failure of \textit{empirical measures} (\textit{Emp}) to converge under this distance. Kernel smoothing and adapted projection have been introduced to construct converging substitutes of empirical measures, known respectively as \textit{smoothed empirical measures} ($\mathsf{S}$-\textit{Emp}) and \textit{adapted empirical measures} ($\mathsf{A}$-\textit{Emp}). However, both approaches have limitations. Specifically, $\mathsf{S}$-\textit{Emp} lack comprehensive convergence results, whereas $\mathsf{A}$-\textit{Emp} in practical applications lead to fewer distinct samples compared to standard empirical measures.

In this work, we address both of the aforementioned issues. First, we develop comprehensive convergence results of $\mathsf{S}$-\textit{Emp}. We then introduce a smoothed version for $\mathsf{A}$-\textit{Emp}, which provide as many
distinct samples as desired. We refer them as $\mathsf{AS}$-\textit{Emp} and establish their convergence in mean, deviation and almost sure convergence. The convergence estimation incorporates two results: the empirical analysis of the \textit{smoothed adapted Wasserstein distance} ($\mathcal{AW}^{(\sigma)}$-distance) and its bandwidth effects. Both results are novel and their proof techniques could be of independent interest.

\end{abstract}

\section{Introduction}
\label{sec:intro}
The development of the \textit{adapted Wasserstein distance} ($\AWD$-distance) is motivated by its robustness in stochastic optimization problems within a dynamic framework, as discussed in \cite{pflug2014multistage}. In stochastic finance, optimal values of various important problems, including pricing and hedging problems, optimal stopping problems, etc., are not continuous with respect to the \textit{Wasserstein distance} ($\WD$-distance). Specifically, two stochastic models can be arbitrarily close to each other under the $\WD$-distance, yet their corresponding optimal values in the aforementioned optimization problems differ significantly. However, when two models are close under the $\AWD$-distance, their optimal values also remain close. In fact, optimal values are Lipschitz continuous with respect to the $\AWD$-distance \cite{backhoff2020adapted}. This implies that the $\AWD$-distance is strong enough to guarantee the robustness of path-dependent problems. Meanwhile, the topology induced by the $\AWD$-distance is already the coarsest topology which makes optimal stopping values continuous \cite{BackhoffVeraguas2019AllAT}. Therefore, the $\AWD$-distance seems to be the appropriate metric when considering stochastic optimization problems under general probability distributions. For further details, please refer to \cite{backhoff2020adapted,backhoff2017causal,lassalle2018causal,ruschendorf1985wasserstein, BackhoffVeraguas2019AllAT, Pammer2022ANO}.

Motivated by the robustness of the $\AWD$-distance, we study the convergence of \textit{empirical measures} (\textit{Emp}) under this metric. Surprisingly, the $\AWD$-distance is so stringent that even empirical measures fail to converge to the underlying measure. To address this issue, two alternatives of empirical measures have been proposed. The first is \textit{smoothed empirical measures} ($\mathsf{S}$-\textit{Emp}), introduced by Pflug-Pichler in \cite{pflug2016empirical}, which convolute empirical measures with smooth kernels. The second is \textit{adapted empirical measures} ($\mathsf{A}$-\textit{Emp}), introduced by Backhoff et al. in \cite{backhoff2022estimating}, by projecting empirical measures onto a grid. However, each of these methods has limitations, either in terms of theoretical convergence results or practical applications.

For $\mathsf{S}$-\textit{Emp}, Pflug-Pichler \cite{pflug2016empirical} established convergence only in probability, under strong assumptions on underlying measures and kernels. Specifically, underlying measures must be compactly supported, with sufficiently regular densities that are bounded away from zero, and have uniform Lipschitz conditional distributions. Additionally, the smoothing kernels must be non-negative, compactly supported, and satisfy uniform consistency conditions. Pflug-Pichler’s proof relies on applying the Nadaraya-Watson estimator recursively to estimate kernel densities. This approach poses significant challenges for establishing convergence in terms of mean and deviation, and almost sure convergence, due to constraints associated with recursive density estimation. As a result, convergence in terms of mean and deviation, and almost sure convergence have largely remained an open problem.

For $\mathsf{A}$-\textit{Emp}, Backhoff et al. \cite{backhoff2022estimating} established convergence in terms of mean and deviation, and almost sure convergence, assuming compactly supported underlying measures. This result was subsequently generalized to unbounded measures on $\R^{dT}$ by \cite{acciaio2022convergence}. However, a practical challenge with $\mathsf{A}$-\textit{Emp} is the issue of ``sample-collapsing", where $\mathsf{A}$-\textit{Emp} yield fewer distinct samples compared to empirical measures after projecting samples onto the grid. To address this problem, it is a common practice to add independent noise to samples before the adapted projection. This heuristic, referred to as \textit{adapted smoothed empirical measures} ($\mathsf{AS}$-\textit{Emp}), can generate as many distinct samples as needed. However, convergence properties of $\mathsf{AS}$-\textit{Emp} have not yet been studied.

The aim of this paper is to conduct a thorough convergence study of both $\mathsf{S}$-\textit{Emp} and $\mathsf{AS}$-\textit{Emp}. In summary, our contributions are as follows:
\begin{itemize}
    \item We establish convergence of $\mathsf{S}$-\textit{Emp} in terms of mean and deviation, and almost sure convergence of $\mathsf{S}$-\textit{Emp}, under notably milder assumptions (\cref{thm:smooth_emp}). Our proof is based on two key results on the \textit{smoothed adapted Wasserstein distance} ($\AWDsigma$-distance): the empirical study of the $\AWDsigma$-distance (\cref{thm:smooth_adw_mean_deviation} and \cref{thm:smooth_adw_as}) and its bandwidth effect (\cref{thm:bw_lip_limit} and \cref{thm:bw_limit}). Both results are significant, and are by themselves worth further investigation.

    \item We formally introduce $\mathsf{AS}$-\textit{Emp}, which is the first variant of empirical measures that is (i) discretely supported as a sum of Dirac measures, (ii) producing distinct samples as many as wanted, and (iii) converging under the $\AWD$-distance. The convergence in terms of mean and deviation and almost sure convergence are established (\cref{thm:ad_smooth_emp_mean_deviation}).

    \item We bound the $\AWD$-distance by the weighted total variation distance for unbounded measures (\cref{thm:metric_dom_new}), which is a key technique used in the main proof.
\end{itemize}

\paragraph{Related Literature.}

Convergence of empirical measures plays a pivotal role in empirical analysis. In the literature, much effort has been devoted to the analysis of empirical measures under $\WD$-distance, see e.g. \cite{bolley2007quantitative,dedecker2015deviation,fournier2015rate,gozlan2007large,lei2020convergence,boissard2011simple,dereich2013constructive,boissard2014mean}. The moment convergence rates are proved in \cite{boissard2014mean} based on iterative trees, and in \cite{dereich2013constructive} based on a so-called Pierce-type estimate. Later, Fournier et al. prove sharp convergence rates in \cite{fournier2015rate}.

To overcome the curse of dimensionality (CoD) rates, recent work has proposed smoothing empirical measures to obtain dimension-free rates under the $\WD$-distance, see e.g. \cite{goldfeld2024limit, chen2021asymptotics, goldfeld2020asymptotic, nietert2021smooth, goldfeld2020gaussian, goldfeld2020convergence}. Motivated by this, in this work we study empirical convergence under the $\AWDsigma$-distance and analogously prove dimension-free rates. \cite{blanchet2024empirical} introduce the so-called \textit{smoothed martingale projection distance} and obtain dimension-free convergence rates under it. Recently, the $\AWDsigma$-distance has been applied in \cite{blanchet2024bounding} to bound the $\AWD$-distance with the $\WD$-distance. Notably, Eckstein–Pammer in \cite{eckstein2024computational} bound the $\AWD$-distance with the total variation distance in compact spaces. However, this is not true for general unbounded measures. To amend this, we introduce the weight total variation distance and use it to bound the $\AWD$-distance in the case of unbounded measures.

The $\AWD$-distance is first investigated in \cite{pflug2012distance, pflug2014multistage} and the convergence of empirical measures under $\AWD$-distance is studied by Pflug-Pichler in \cite{pflug2016empirical}. They notice that empirical measures fail to converge under $\AWD$-distance, so they introduce $\mathsf{S}$-\textit{Emp} and prove convergence. Later, Backhoff et al. in \cite{backhoff2022estimating} introduce $\mathsf{A}$-\textit{Emp} on compact spaces and prove convergence. Recently, Acciaio-Hou extend convergence results in \cite{backhoff2022estimating} to general measures in \cite{acciaio2022convergence}.\\ 

\paragraph{Organization of the paper.}
In \cref{sec:intro}, we give a brief introduction of the problem and elaborate our contributions. Then, in \cref{sec:setteing}, we introduce the setting and state our main results. In \cref{sec:bounding}, we prove domination inequality between the weighted total variation distance and the adapted Wasserstein distance. In \cref{sec:smooth_distance}, we introduce smooth distances and prove the convergence of empirical measures under various smooth distances. In \cref{sec:bw_effect}, we analyze the bandwidth effect. In \cref{sec:smooth_emp}, we prove the convergence of smoothed empirical measures. In \cref{sec:smooth_and_ad_emp_meas}, we prove the convergence of adapted smoothed empirical measures. Finally, in \cref{sec:appendix}, we collect some technical results and needed tools.\\

\paragraph{Notations.} 
Throughout the paper, we let $d \geq 1$ be the dimension of the state space and $T \geq 1$ be the time horizon. Let $\mathcal{P}(\R^{dT})$ be the space of canonical Borel probability measures on $\R^{dT}$, and let $\mu, \nu \in \mathcal{P}(\R^{dT})$. We consider finite discrete-time paths $x = x_{1:T} = (x_1,\dots,x_T) \in \R^{dT}$, where $x_{t} \in \R^{d}$ represents the value of the path at time $t = 1,\dots, T$. For $t=1,...,T$, we denote by $x_{1:t}=(x_1,...,x_t) \in \R^{dt}$ and equip $\R^{dt}$ with a sum-norm $\Vert \cdot \Vert \colon \R^{dt} \to \R$ defined by $\Vert x \Vert = \sum_{s = 1}^{t} \Vert x_{s} \Vert_{\R^{d}}$. 

For $\mu \in \mathcal{P}(\R^{dT})$, we denote the up to time $t$ marginal of $\mu$ by $\mu_{1:t}$, and the kernel (disintegration) of $\mu$ w.r.t. $x_{1:t}$ by $\mu_{x_{1:t}}$, so the following holds: $\mu(dx_{t+1}) = \int_{\R^{dt}}\mu_{x_{1:t}}(dx_{t+1})\mu_{1:t}(dx_{1:t})$. Similarly, we denote the up to time $t$ marginal of $\pi\in\cpl(\mu,\nu)$ by $\pi_{1:t}$, and the kernel of $\pi$ w.r.t. $(x_{1:t}, y_{1:t})$ by $\pi_{x_{1:t}, y_{1:t}}$, so that $\pi(dx_{t+1},dy_{t+1}) = \int_{\R^{dt} \times \R^{dt}}\pi_{x_{1:t},y_{1:t}}(dx_{t+1},dy_{t+1})\pi_{1:t}(dx_{1:t},dy_{1:t})$. For simplicity, we denote by $\mu_1 = \mu_{1:1}$ and $\pi_1 = \pi_{1:1}$.

For $p\geq 1$, we denote by $M_{p}(\mu) = \int \Vert x\Vert^{p}\mu(dx)$ the $p$-th moment of $\mu$ and denote by $\mathcal{P}_{p}(\R^{dT})$ probability measures on $\R^{dT}$ with finite $p$-th moments. For $\alpha,\gamma > 0$, we denote by $\mathcal{E}_{\alpha,\gamma}(\mu) = \int \exp\left(\gamma\Vert x\Vert^{\alpha}\right)\mu(dx)$ the $(\alpha,\gamma)$-exponential moment of $\mu$. We let $(X^{(n)})_{n\in\N}$ be i.i.d. samples from $\mu$ defined on some probability space $(\Omega, \mathcal{F},\sP)$. For all $k\in \N$ and $\sigma > 0$, we denote by $\calN_{\sigma, k} = \calN(\mathbf{0},\sigma^2  \mathbf{I}_{k})$ the Gaussian distribution and by $\varphi_{\sigma, k}$ its density function. In particualr when $k=dT$, for all $\sigma > 0$, we denote by $\calN_{\sigma} = \calN(\mathbf{0},\sigma^2  \mathbf{I}_{dT})$ the Gaussian distribution and by $\varphi_{\sigma}$ its density function. For all $\mu \in \calP(\R^{dT})$, we call the convolution measure of $\mu$ and $\calN_\sigma$ the Gaussian-smoothed measure of $\mu$ such that $(\mu \ast \calN_{\sigma})(dx) = \int_{\R^{dT}}\varphi_{\sigma}(x-y)\mu(dy)dx,$
and we denote $\mu \ast \calN_{\sigma}$ by $\mu_\sigma$.

\section{Main results}\label{sec:setteing}
\begin{definition}[Wasserstein distance]
    For $\mu,\nu\in\calP_1(\R^{dT})$, the first order \textit{Wasserstein distance} $\WD_1(\cdot,\cdot)$ on $\mathcal{P}_1(\R^{dT})$ is defined by 
    \begin{equation*}
        \WD_1(\mu,\nu) = \inf_{\pi \in \cpl(\mu,\nu)}\int \| x - y\| \,\pi(dx,dy),
    \end{equation*}
    where $\cpl(\mu,\nu)$ denotes the set of couplings between $\mu$ and $\nu$, that is, probabilities in $\mathcal{P}(\R^{dT}\times \R^{dT})$ with first marginal $\mu$ and second marginal $\nu$. 
\end{definition}

\begin{definition}[Weighted total variation distance]
    For $\mu,\nu\in\calP_1(\R^{dT})$, the first order \textit{weighed total variation distance} $\TVD_1(\cdot,\cdot)$ on $\mathcal{P}(\R^{dT})$ is defined by 
    \begin{equation*}
        \TVD_{1}(\mu,\nu) = \int (\Vert x \Vert + \frac{1}{2}) |\mu - \nu|(dx) = \inf_{\pi \in \cpl(\mu,\nu)}\int (\Vert x \Vert + \Vert y \Vert + 1)\mathbbm{1}_{\{x\neq y\}}\pi(dx,dy),
    \end{equation*}
    where $|\mu - \nu| = \mu + \nu - 2(\mu \wedge \nu)$ {\color{black}and $\mu\wedge\nu = \mu - (\mu - \nu)^+ = \nu - (\nu - \mu)^+$}.
\end{definition}
 The definition of weighted total variation distance $\TVD_1(\mu,\nu)$ is motivated by the primal formulation of classical total variation distance $\TVD(\mu,\nu) = \inf_{\pi\in\cpl(\mu,\nu)}\int \mathbbm{1}_{\{x\neq y\}}\pi(dx,dy)$, augmented with an additional cost $\Vert x \Vert + \Vert y \Vert$. This modification ensures that $\TVD_1(\mu,\nu)$ serves as an upper bound for $\WD_1(\mu,\nu)$ for all $\mu,\nu \in \calP_1(\R^{dT})$.
Next, we restrict our attention to couplings $\pi \in \cpl(\mu,\nu)$ such that the conditional law of $\pi$ is still a coupling of the conditional laws of $\mu$ and $\nu$, that is, $\pi_{x_{1:t},y_{1:t}} \in \cpl\big(\mu_{x_{1:t}}, \nu_{y_{1:t}}\big)$. Such couplings are called
bi-causal, and denoted by $\bccpl(\mu,\nu)$.
The causality constraint can be expressed in different equivalent ways; see e.g. \cite{backhoff2017causal,acciaio2020causal,bremaud1978changes}. Roughly, in a causal transport, for every time $t$, only information on the $x$-coordinate up to time $t$ is used to determine the mass transported to the $y$-coordinate at time $t$.
And in a bi-causal transport this holds in both directions, i.e. also when exchanging the role of $x$ and $y$.

\begin{definition}[Adapted Wasserstein distance]
     For $\mu,\nu\in\calP_1(\R^{dT})$, the first order \textit{adapted Wasserstein distance} $\AWD_1(\cdot,\cdot)$ on $\mathcal{P}(\R^{dT})$ is defined by 
    \begin{equation}
    \label{eq:adaptedwd} 
        \AWD_1(\mu,\nu) = \inf_{\pi \in \bccpl(\mu,\nu)}\int\|x-y\|\,\pi(dx,dy).
    \end{equation}
\end{definition}

Pflug-Pichler refer to the adapted Wasserstein distance as nested distance in \cite{pflug2014multistage}, with an alternative representation through a dynamic programming principle by disintegrating \eqref{eq:adaptedwd} and replacing conditional laws with $\pi_{x_{1:t},y_{1:t}} \in \cpl\big(\mu_{x_{1:t}}, \nu_{y_{1:t}}\big)$. For notational simplicity, we state it here only for the case $t = 1$, where one obtains the representation 
\begin{equation}
\label{eq:dpp}
    \AWD_1(\mu,\nu) = \inf_{\pi_1 \in \cpl(\mu_1,\nu_1)}\int\Vert x_1 - y_1\Vert_{\R^{d}} + \AWD_1(\mu_{x_1},\nu_{y_1})\pi_1(dx_1,dy_1).
\end{equation}
This reflects clearly that $\AWD$ considers not only marginal laws but also the difference between conditional laws. The example below explicitly shows the gap between Wasserstein distance and adapted Wasserstein distance, when conditional laws mismatch. Additionally, when regarding $\mu$ and $\nu$ as distributions of risky assets, it clearly illustrates the inappropriateness of the Wasserstein distance to gauge closeness of financial markets, and the way in which its adapted counterpart amends to it. 
\begin{example}
\label{ex:classical_example}
    Let $\mu, \nu \in \mathcal{P}([0,1]^2)$ be given by $\mu = \frac{1}{2}\delta_{(0,1)} + \frac{1}{2}\delta_{(0,-1)}$ and $\nu = \frac{1}{2}\delta_{(\epsilon,1)} + \frac{1}{2}\delta_{(-\epsilon,-1)}$, with $\epsilon \in (0, 1)$.\footnote{ We denote by $\delta_x$ the Dirac measure concentrated at $x \in \R^{dT}$ for all $d,T \in \N$.} On one hand, we have $\WD_1(\mu,\nu) = \epsilon$ by optimally coupling $(0,1)$ with $(\epsilon, 1)$ and $(0,-1)$ with $(-\epsilon, 1)$. On the other hand, since $\AWD_1(\mu_{x_1},\nu_{y_1}) = \WD_1(\mu_{x_1},\nu_{y_1}) = 1$ for all $x_1,y_1\in\R$, thus by \eqref{eq:dpp}, we get
    \[
    \AWD_1(\mu,\nu) = \inf_{\pi_1 \in \cpl(\mu_1,\nu_1)}\int \Vert x_1 - y_1\Vert \pi_1(dx_1, dy_1) + 1 = 1 + \epsilon.
    \]
    Therefore, by letting $\epsilon \to 0$, we get $\lim_{\epsilon \to 0}\frac{\AWD_{1}(\mu,\nu)}{\WD_1(\mu,\nu)} = \infty$, which indicates that the topology induced by $\AWD_1(\cdot, \cdot)$ is stricter than the weak topology induced by $\WD_1(\cdot, \cdot)$.
\end{example}   
 In the above example, let us consider a financial market with an asset whose law is described by $\mu$, and another market with an asset whose law is described by $\nu$. Then under the Wasserstein distance the two markets are judged as being close to each other, while they clearly present very different features (random versus deterministic evolution, no-arbitrage versus arbitrage, etc.). It is also evident how optimization problems in the two situations would lead to very different decision making. This is a standard example to motivate the introduction of adapted distances, that instead can distinguish between the two models.  

\subsection{Smoothed empirical measures}
In this subsection, we present convergence of smoothed empirical measures.

\begin{definition}[Empirical measures]
\label{def:emp_mea}
    For $\mu \in \calP_1(\R^{dT})$ and $N\in\N$, we denote by $\mu^{N} \coloneqq \frac{1}{N}\sum_{n=1}^{N}\delta_{X^{(n)}}$ the \textit{empirical measures} of $\mu$, where $(X^{(n)})_{n\in\N}$ are i.i.d. samples from $\mu$.
\end{definition}

\begin{definition}[Smoothed empirical measures]
\label{def:smooth_emp_mea}
    For $\mu \in \calP_1(\R^{dT})$, $N\in\N$ and $\sigma >0$, we call the convoluted measures of empirical measures $\mu^{N}$ and $\calN_{\sigma}$ \textit{smoothed empirical measures} ($\mathsf{S}$-\textit{Emp}) of $\mu$, denoted by $\mu^{N} \ast \calN_{\sigma}$.
\end{definition}
On one hand, with the bandwidth $\sigma$ fixed, $\mu^N \ast \calN_{\sigma}$ converges to $\mu \ast \calN_\sigma$ under $\AWD$-distance in terms of mean, deviation and converge almost surely; see \cref{sec:smooth_distance_awd}. On the other hand, as the bandwidth $\sigma$ goes to $0$, $\AWD_1(\mu, \mu \ast \mathcal{N}_{\sigma})$ converge linearly w.r.t. $\sigma$ if $\mu$ has Lipschitz kernels; see Section~\ref{sec:bw_effect_lip}. Combining both, we establish convergence of $\mu^N \ast \calN_{\sigma_N}$ to $\mu$; see \cref{sec:smooth_emp} for proofs.
\begin{definition}
    Let $L>0$. We say that $\mu\in\mathcal{P}_1(\R^{dT})$ has $L$-Lipschitz kernels if it admits a disintegration s.t. for all $t=1,\dots,T-1$, $x_{1:t} \mapsto \mu_{x_{1:t}}$ is $L$-Lipschitz (where $\calP(\R^{d})$ is equipped with $\WD_1$).
\end{definition}
\begin{theorem}\label{thm:smooth_emp}
Let $\mu \in \calP_1(\R^{dT})$, $\sigma_N = N^{-r}$ for all $N\in\N$ where $r=(dT+2)^{-1}$. Then
\begin{equation*}
    \lim_{N \to \infty}\AWD_{1}(\mu , \mu^N \ast \calN_{\sigma_N} )  = 0, \quad \sP\text{-a.s.}
\end{equation*}
In addition, assume $K \subseteq \R^{dT}$ compact, $L>0$ and $\mu \in \calP(K)$ with $L$-Lipschitz kernels. Then there exist $c, C>0$ depending only on $d, T, L, K$ s.t. for all $x > 0$ and $N \in \N$,
\begin{equation}
\label{eq:thm:smooth_emp:0.1}
    \E\big[\AWD_{1}(\mu , \mu^N \ast \calN_{\sigma_N} ) \big] \leq C N^{-r},
\end{equation}
\begin{equation}
\label{eq:thm:smooth_emp:0.2}
    \sP\Big(\AWD_{1}(\mu , \mu^N \ast \calN_{\sigma_N} ) \geq x + C N^{-r} \Big) \leq e^{-\frac{x^2 N^{1-r} }{c}}.
\end{equation}
\end{theorem}
\begin{remark}
    Recall that the optimal mean convergence rates of empirical measures under $\WD$-distance are
    \begin{equation*}
        \E\big[\WD_{1}(\mu , \mu^N )\big] \leq C\left\{\begin{aligned}
            &N^{-\frac{1}{2}}, &dT=1,\\
            &N^{-\frac{1}{2}}\log(N+1), &dT=2,\\
            &N^{-\frac{1}{dT}}, &dT\geq 3,
        \end{aligned}\right.
    \end{equation*}
    which is slightly faster than the $O(N^{-\frac{1}{dT+2}})$ convergence rate of $\mathsf{S}$-\textit{Emp} under $\AWD$-distance  in \cref{thm:smooth_emp}. However, the same $O(N^{-\frac{1}{dT}})$ convergence rate of $\mathsf{A}$-\textit{Emp} under the $\AWD$-distance is established in \cite{backhoff2022estimating, acciaio2022convergence} when $dT\geq 3$. This naturally raises the question: \textit{where does this gap come from?} As we shall see in the proof, the rate in \cref{thm:smooth_emp} essentially relies on those in \cref{thm:smooth_adw_mean_deviation} and \cref{thm:bw_lip_limit}.
    Although the rate $O(N^{-\frac{1}{2}})$ in \cref{thm:smooth_adw_mean_deviation} is sharp with respect to $N$ (i.e., it matches the Monte Carlo rate), the constant in front only scales as $O(\sigma^{-\frac{dT}{2}})$ with respect to $\sigma$. Therefore, when this is combined with the $O(\sigma)$ bandwidth effect in \cref{thm:bw_lip_limit}, we obtain the overall rate $O(N^{-\frac{1}{dT+2}})$ by taking $\sigma = N^{-\frac{1}{dT+2}}$. Notably, this gap between is not unique to the $\AWD$-distance. Similar phenomena have been observed in the context of the classical Wasserstein distance; see \cite{horowitz1994mean, cosso2025mean} for detailed discussions and related results.
\end{remark}

\subsection{Adapted smoothed empirical measures} 
In this subsection, we present convergence of adapted smoothed empirical measures. First, let us recall the definition of adapted empirical measures from \cite{acciaio2022convergence}.
\begin{definition}[Adapted empirical measures]\label{def:uniform_adapted_empirical_measure}
    For $\mu \in \calP_1(\R^{dT})$, $N \in \N$, and grid size $\Delta_N > 0$, we let $G_N = \lceil \frac{1}{\Delta_N}\rceil$ and consider the uniform partition $\hat{\Phi}^N$ of $\R^{dT}$ given by
    \begin{equation*}
        \hat{\Phi}^{N} = \left\{\hat{\mathcal{C}}_{\mathbf{z}}^{N} = \Big[0,\frac{1}{G_N}\Big]^{dT} + \frac{1}{G_N} \mathbf{z} , \mathbf{z} \in \mathbb{Z}^{dT}\right\}.
    \end{equation*}
    Let $\hat{\Lambda}^N$ be the set of mid points of all cubes $\hat{\mathcal{C}}_{\mathbf{z}}^{N}$ in the partition $\hat{\Phi}^{N}$, and let $\hat{\varphi}^{N}\colon \R^{dT} \to \hat{\Lambda}^N$ map each cube $\hat{\mathcal{C}}_{\mathbf{z}}^{N}$ to its mid point (points belonging to more than one cube can be mapped into any of them).
    Then we denote by 
    $$\hat{\mu}^{N} = \frac{1}{N}\sum_{n=1}^{N}\delta_{\hat{\varphi}^{N}(X^{(n)})}$$ 
    \textit{adapted empirical measures} ($\mathsf{A}$-\textit{Emp}) of $\mu$ with grid size $\Delta_N$.
\end{definition}
\begin{remark}
    Intuitively, $\mathsf{A}$-\textit{Emp} is constructed via the following procedure: (i) we tile $\R^{dT}$ with cubes of size $(\frac{1}{G_N})^{dT}$ that form the partition $\hat{\Phi}^{N}$; (ii) we project all points in each cube $\hat{\mathcal{C}}_{\mathbf{z}}^{N}$ to its mid point. As a result, the push-forward measures obtained as empirical measures of the samples after projections are precisely $\mathsf{A}$-\textit{Emp}. 
\end{remark}
Since adapted projection maps samples onto grid points, $\mathsf{A}$-\textit{Emp} have less distinct samples than empirical measures, as different samples may be projected to the same path on grid. Motivated by the idea of data augmentation, which is first proposed in \cite{simard2003best} to perturb existing data to create new examples, we introduce adapted smoothed empirical measures by adding independent Gaussian noise to samples, and subsequently applying the adapted projection introduced in \cref{def:uniform_adapted_empirical_measure}.

\begin{definition}[Adapted smoothed empirical  measures]\label{def:ad_smooth_emp}
    For $\mu \in \calP_1(\R^{dT})$, $\sigma >0$, $N, M \in \N$ and grid size $\Delta_N > 0$, we let $\zeta = (\zeta^m)_{m=1}^N$ where $\zeta^m$ be distinct points in $(0,\frac{1}{2G_N})^{dT}$ and denote by 
    \begin{equation*}
        \tilde{\mu}^{N,M}_{\sigma,\zeta} \defeq \frac{1}{M}\sum_{m=1}^{M}\tilde{\mu}^N_{\zeta,m}
    \end{equation*}
    \textit{adapted smoothed empirical measures} ($\mathsf{AS}$-\textit{Emp}) of $\mu$, where 
    \begin{equation*}
        \tilde{\mu}^N_{\zeta,m} = (x\mapsto x+\zeta^m)_{\#}\tilde{\mu}^N_m,\quad \tilde{\mu}^N_m\defeq \frac{1}{N}\sum_{n=1}^N \delta_{\hat{\varphi}^{N}(X^{(n)} + \sigma \varepsilon^{(n,m)})},
    \end{equation*}
    $(\varepsilon^{(n,m)})_{n,m\in\N}$ are i.i.d. samples from $\calN_{1}$ and $\hat{\varphi}^{N}$ is the adapted projection with grid size $\Delta_N$ in \cref{def:uniform_adapted_empirical_measure}. In particular, when $M=1$, we set w.l.o.g. $\zeta^1 = 0$ and denote by $\hat{\mu}_{\sigma}^{N} = \tilde{\mu}_{\sigma,\zeta}^{N,1}$ \textit{adapted smoothed 1-empirical measures} ($\mathsf{AS1}$-\textit{Emp}).
\end{definition}

\begin{remark}
    Intuitively, $\mathsf{AS}$-\textit{Emp} add noise to samples and then project the noised samples on adapted grid. The adapted projection is necessary; without it, $\mathsf{AS}$-\textit{Emp} fail to converge. The introduction of $\zeta^m$, $m=1,\dots,M$ is more technical, due to the non-convexity of the $\AWD$-distance; see details in \cref{sec:ad_smooth_emp_meas}.
\end{remark}

\begin{figure}[H]
    \centering
    \includegraphics[width=\linewidth]{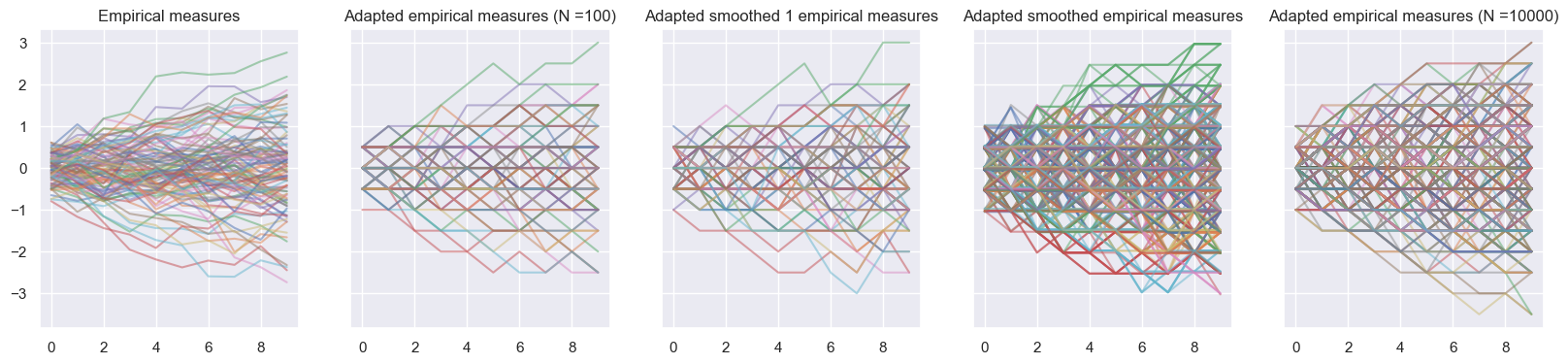}
    \caption{Visualization of different empirical measures. From left to right, they are empirical measures $\mu^N$ ($N=100$), adapted empirical measures $\hat{\mu}^N$ ($N=100$),  adapted smoothed 1-empirical measures $\tilde{\mu}^{N,1}_{\sigma}$ ($N=100$),  adapted smoothed empirical measures $\tilde{\mu}_{\sigma,\zeta}^{N,M}$ ($N=100, M=10$), and adapted empirical measures $\hat{\mu}^N$ ($N=1000$).}
    \label{fig:compare_different_measures}
\end{figure}

Notably, by adding Gaussian noise, $\mathsf{AS}$-\textit{Emp} are able to provide as many distinct samples as wanted. We call this property data-augmentation. Hence, $\mathsf{AS}$-\textit{Emp} enrich the support of samples compared to $\mathsf{A}$-\textit{Emp}; see Figure~\ref{fig:compare_different_measures}. We further establish convergence of $\mathsf{AS}$-\textit{Emp} in terms of mean and deviation, and almost sure convergence; see \cref{sec:ad_smooth_emp_meas} for the proof.
\begin{theorem}
\label{thm:ad_smooth_emp_mean_deviation}
    Set  $\Delta_N = \sigma_N = N^{-\frac{1}{\calD(d)T}}$ for all $N\in\N$, with $\calD(d) = d$ if $d\geq 3$ and  $\calD(d) = d+1$ if $d=1,2$. Let $L >0$, $\alpha \geq 2$, $\gamma > 0$, $\mu \in \mathcal{P}_1(\R^{dT})$ with finite $(\alpha,\gamma)$-exponential moment. Assume that $\bar \calE_{\alpha,\gamma}(\mu) \coloneqq \sup_{x_{1:t}\in\R^{dt}}\calE_{\alpha,2^\alpha \gamma}(\mu_{x_{1:t}}) < \infty$ for all $ t=1,\dots,T-1$, and that for all $\sigma \in (0,1]$, $\mu \ast \calN_{\sigma}$ has $L$-Lipschitz kernels. Then there exist constants $c,C > 0$ {\color{black} depending only on $d, T, L, \alpha, \gamma, \mathcal{E}_{\alpha,\gamma}(\mu), \bar{\calE}_{\alpha,\gamma}(\mu)$} s.t., for all $x>0$ and $N, M\in \N$,
    \begin{equation}
    \label{eq:thm:ad_smooth_emp_mean_deviation:0}
    \E\big[\AWD_1(\mu,{\color{black}\tilde{\mu}_{\sigma_N,\zeta}^{N,M}})\big] \leq CN^{-\frac{1}{\calD(d)T}}, \quad \sP\Big(\AWD_1(\mu,{\color{black}\tilde{\mu}_{\sigma_N,\zeta}^{N,M}}) \geq  x + CN^{-\frac{1}{\calD(d)T}} \Big) \leq CMe^{-c Nx^2},
    \end{equation}
    and $\lim_{N \to \infty}\AWD_1(\mu,{\color{black}\tilde{\mu}_{\sigma_N,\zeta}^{N,M}}) = 0~\sP\text{-a.s.}$.
\end{theorem}

To the best of our knowledge, $\mathsf{AS}$-\textit{Emp} are the first variants of empirical measures that are (i) discretely supported as a sum of Dirac measures, (ii) producing as many distinct samples as wanted, and (iii) converging under $\AWD$-distance; see Table~\ref{table:emp_meas} for comparison.
\begin{table}[H]
\centering
\begin{tabular}{lllcc}
\toprule
Symbol & Name & Convergence ($\AWD_1$) & Discrete & Augment data\\
  \midrule
  $\mu^N$ & empirical measures (\textit{Emp}) & \xmark & \cmark & \xmark \\
  $\mu_{\sigma_N} $ & smoothed measures & \cmark ~~(\cref{sec:bw_effect}) & \xmark & \xmark \\[1pt]
  $\mu^N \ast \calN_{\sigma_N} $ & $\mathsf{S}$-\textit{Emp} & \cmark ~~(\cref{sec:smooth_emp}) & \xmark & \cmark \\[1pt]
  $\hat{\mu}^N $ & $\mathsf{A}$-\textit{Emp} & \cmark ~~(\cite{acciaio2022convergence}) & \cmark & \xmark \\[1pt]
$\hat{\mu}_{\sigma_N}^{N}$ & $\mathsf{AS1}$-\textit{Emp} & \cmark~~(\cref{sec:ad_emp_smooth_meas}) & \cmark & \xmark \\[3pt]
  ${\color{black}\tilde{\mu}_{\sigma_N,\zeta}^{N,M}}$ & $\mathsf{AS}$-\textit{Emp} & \cmark~~(\cref{sec:ad_smooth_emp_meas}) & \cmark & \cmark \\
  \bottomrule
\end{tabular}
\caption{Comparison of different empirical measures.}
\label{table:emp_meas}
\end{table}

\section{Bounding $\AWD_1$ by $\TVD_1$}
\label{sec:bounding}
In this section, we prove inequality between $\AWD_1$ and $\TVD_1$ for unbounded measures, which will serve as a key bridge in the proof of our main results later.
\begin{definition}[Weighted adapted total variation distance]
    For $\mu,\nu\in\calP_1(\R^{dT})$, the first order \textit{weighed adapted total variation distance} $\AVD_1(\cdot,\cdot)$ on $\mathcal{P}(\R^{dT})$ is defined by 
    \begin{equation*}
        \AVD_1(\mu,\nu) = \inf_{\pi^{\mathrm{bc}} \in \bccpl(\mu,\nu)}\int (\Vert x \Vert + \Vert y \Vert + 1)\mathbbm{1}_{\{x\neq y\}}\pi(dx,dy).
    \end{equation*}
\end{definition}

\begin{definition}[Linear conditional moments]
\label{def:linear_cond_moment}
    For $\alpha >0$, we say that $\mu \in \calP_1(\R^{dT})$ has $\alpha$-linear conditional moments if for all $t = 1,\dots,T-1$ and $x_{1:t} \in \R^{dt}$, $\int \Vert x_{t+1} \Vert d\mu_{x_{1:t}} \leq \alpha (\Vert x_{1:t} \Vert + 1)$.
\end{definition}
{\color{black}
First, we recall the Kantorovich duality of optimal transport, a key lemma in the proof of \cref{lem:av}; see \cite[Theorem~5.10]{villani2009optimal} for detailed statements and the proof.
\begin{theorem}[Kantorovich duality]
\label{thm:kan_dual}
Let $\mu,\nu\in\calP(\R^d)$ and $\kappa\colon \R^d \times\R^d \to \R_{\geq 0}$ be a non-negative lower semicontinuous cost function. Then the following duality holds:
\begin{equation*}
\begin{split}
    \mathrm{OT}_\kappa(\mu,\nu) \coloneqq \inf_{\pi \in \cpl(\mu,\nu)}\int \kappa(x,y)\pi(dx,dy) = \sup_{\substack{f \in C_{b}(\R^d), g \in C_{b}(\R^{d})\\ f(x) + g(y) \leq \kappa(x,y)}} \Big(\int f(x)\mu(dx) + \int g(y) \nu(dy)\Big).
\end{split}
\end{equation*}
If there exists $(\kappa_1,\kappa_2) \in L^1(\mu)\times L^1(\nu)$ such that for all $(x,y) \in \R^d \times \R^d$, $\kappa(x,y) \leq \kappa_1(x) + \kappa_2(y)$, then both the primal and dual Kantorovich problems are attainable. 
\end{theorem}

}
\begin{lemma}
\label{lem:av}
    Let $\alpha > 0$ and $\mu,\nu\in\calP_1(\R^{dT})$ with $\alpha$-linear conditional moments. Then for all $t = 1,\dots, T-1$,
    \begin{equation*}
        \AVD_1(\mu_{1:t+1},\nu_{1:t+1}) \leq (2 + 4\alpha)\AVD_1(\mu_{1:t}, \nu_{1:t}) + \TVD_1(\mu,\nu).
    \end{equation*}
\end{lemma}
\begin{proof}
Since $\mu$ and $\nu$ have $\alpha$-linear conditional moments, there exists $c>0$ such that for all $t = 1,\dots,T-1$ and $x_{1:t}, y_{1:t} \in \R^{dt}$,
\begin{equation}
\label{eq:lem_av:0}
    \int \Vert x_{t+1} \Vert d\mu_{x_{1:t}}  + \int \Vert y_{t+1} \Vert d\nu_{y_{1:t}}\leq \beta (\Vert x_{1:t} \Vert + \Vert y_{1:t} \Vert + 1),
\end{equation}
where $\beta = 2\alpha$. For notational simplicity, throughout the proof, we denote by $c(x,y) = \Vert x \Vert + \Vert y \Vert + 1$. Notice that for all $\pi \in \bccpl(\mu_{1:t+1},\nu_{1:t+1})$, we can separate the cases for indicator function in the cost as
\begin{equation}
\label{eq:lem_av:1}
    \begin{split}
    &\quad~\int c(x_{1:t+1},y_{1:t+1})\mathbbm{1}_{\{x_{1:t+1}\neq y_{1:t+1}\}}d\pi \\ 
    &= \int c(x_{1:t+1},y_{1:t+1})\mathbbm{1}_{\{x_{1:t}\neq y_{1:t}\}}d\pi + \int c(x_{1:t+1},y_{1:t+1})\mathbbm{1}_{\{x_{t+1}\neq y_{t+1}\}}\mathbbm{1}_{\{x_{1:t}= y_{1:t}\}}d\pi\\
    &= \int c(x_{1:t+1},y_{1:t+1})\mathbbm{1}_{\{x_{1:t}\neq y_{1:t}\}}d\pi + \int \int c(x_{1:t+1},y_{1:t+1})\mathbbm{1}_{\{x_{t+1}\neq y_{t+1}\}} d\pi_{x_{1:t},y_{1:t}}  \mathbbm{1}_{\{x_{1:t}= y_{1:t}\}}d\pi_{1:t}.
    \end{split}
\end{equation}
For the first term of the last line in \eqref{eq:lem_av:1}, we split the cost such that
\begin{equation}
\label{eq:lem_av:2}
    \begin{split}
    &\quad \int c(x_{1:t+1},y_{1:t+1})\mathbbm{1}_{\{x_{1:t}\neq y_{1:t}\}}d\pi\\
    &= \int c(x_{1:t},y_{1:t})\mathbbm{1}_{\{x_{1:t}\neq y_{1:t}\}}d\pi_{1:t} + \int \int (\Vert x_{t+1}\Vert + \Vert y_{t+1}\Vert)d\pi_{x_{1:t},y_{1:t}}\mathbbm{1}_{\{x_{1:t}\neq y_{1:t}\}}d\pi_{1:t}\\
    &= \int c(x_{1:t},y_{1:t})\mathbbm{1}_{\{x_{1:t}\neq y_{1:t}\}}d\pi_{1:t} + \int  (\int \Vert x_{t+1} \Vert d\mu_{x_{1:t}}  + \int \Vert y_{t+1} \Vert d\nu_{y_{1:t}})\mathbbm{1}_{\{x_{1:t}\neq y_{1:t}\}}d\pi_{1:t}\\
    &\leq (1+ \beta)\int c(x_{1:t},y_{1:t})\mathbbm{1}_{\{x_{1:t}\neq y_{1:t}\}}d\pi_{1:t},
    \end{split}
\end{equation}
where the last inequality is by \eqref{eq:lem_av:0}. Plugging \eqref{eq:lem_av:2} back into \eqref{eq:lem_av:1}, we have
\begin{equation}
\label{eq:lem_av:3}
    \begin{split}
    \int c(x_{1:t+1},y_{1:t+1})\mathbbm{1}_{\{x_{1:t+1}\neq y_{1:t+1}\}}&d\pi
    \leq (1+ \beta)\int c(x_{1:t},y_{1:t})\mathbbm{1}_{\{x_{1:t}\neq y_{1:t}\}}d\pi_{1:t}\\
    & + \int \int c(x_{1:t+1},y_{1:t+1})\mathbbm{1}_{\{x_{t+1}\neq y_{t+1}\}} d\pi_{x_{1:t},y_{1:t}}\mathbbm{1}_{\{x_{1:t}= y_{1:t}\}}d\pi_{1:t}.
    \end{split}
\end{equation}
Then by taking infimum of $\pi$ over $\bccpl(\mu_{1:t+1},\nu_{1:t+1})$ in \eqref{eq:lem_av:3} and the definition of $\AVD_1$, we have
\begin{equation}
\label{eq:lem_av:4}
\begin{split}
    &\quad~ \AVD_1(\mu_{1:t+1},\nu_{1:t+1})\\
    &= \inf_{\pi \in \bccpl(\mu_{1:t+1},\nu_{1:t+1})} \int c(x_{1:t+1},y_{1:t+1})\mathbbm{1}_{\{x_{1:t+1}\neq y_{1:t+1}\}}d\pi\\
    &= \inf_{\pi_{1:t} \in \bccpl(\mu_{1:t},\nu_{1:t})} \inf_{\pi_{x_{1:t},y_{1:t}} \in \cpl(\mu_{x_{1:t}},\nu_{y_{1:t}})}\int \int c(x_{1:t+1},y_{1:t+1})\mathbbm{1}_{\{x_{1:t+1}\neq y_{1:t+1}\}}d\pi_{x_{1:t},y_{1:t}}d\pi_{1:t}\\
    &\leq \inf_{\pi_{1:t} \in \bccpl(\mu_{1:t},\nu_{1:t})} \inf_{\pi_{x_{1:t},y_{1:t}} \in \cpl(\mu_{x_{1:t}},\nu_{y_{1:t}})} \bigg[(1+ \beta)\int c(x_{1:t},y_{1:t})\mathbbm{1}_{\{x_{1:t}\neq y_{1:t}\}}d\pi_{1:t}\\
    &\qquad\qquad + \int \int c(x_{1:t+1},y_{1:t+1})\mathbbm{1}_{\{x_{t+1}\neq y_{t+1}\}} d\pi_{x_{1:t},y_{1:t}}\mathbbm{1}_{\{x_{1:t}= y_{1:t}\}}d\pi_{1:t}\bigg]\\
    &= \inf_{\pi_{1:t} \in \bccpl(\mu_{1:t},\nu_{1:t})} \bigg[(1+ \beta)\int c(x_{1:t},y_{1:t})\mathbbm{1}_{\{x_{1:t}\neq y_{1:t}\}}d\pi_{1:t} \\
    &\qquad\qquad + \int\inf_{\pi_{x_{1:t},y_{1:t}} \in \cpl(\mu_{x_{1:t}},\nu_{y_{1:t}})} \int c(x_{1:t+1},y_{1:t+1})\mathbbm{1}_{\{x_{t+1}\neq y_{t+1}\}} d\pi_{x_{1:t},y_{1:t}}\mathbbm{1}_{\{x_{1:t}= y_{1:t}\}}d\pi_{1:t}\bigg].
\end{split}
\end{equation}
{\color{black}It is worth noting that $\bccpl(\cdot, \cdot)$, as a subset of bi-causal couplings, introduces a rather different mechanism when interacting with infimum and integral compared to $\cpl(\cdot, \cdot)$ in \eqref{eq:lem_av:4}, which could be viewed as a dynamic programming principle for adapted optimal transport problems.}
{\color{black} For the second term of the last line in \eqref{eq:lem_av:4}, we notice that for all $x_{1:t},y_{1:t}\in\R^{dt}$ fixed, $c(x_{1:t+1},y_{1:t+1})\mathbbm{1}_{\{x_{t+1} \neq y_{t+1}\}} \eqqcolon \kappa(x_{t+1}, y_{t+1})$ is non-negative and lower semicontinuous w.r.t. $(x_{t+1},y_{t+1})$. Moreover, with $x_{1:t},y_{1:t}\in\R^{dt}$ fixed, by choosing $\kappa_1(x_{t+1}) = \frac{1}{2} + \Vert x_{1:t+1} \Vert$ and $\kappa_2(y_{t+1}) = \frac{1}{2} + \Vert y_{1:t+1} \Vert$, we have $\kappa(x_{t+1}, y_{t+1}) \leq \kappa_1(x_{t+1}) + \kappa_2(y_{t+1})$ and $(\kappa_1, \kappa_2) \in L^1(\mu_{x_{1:t}})\times L^1(\nu_{y_{1:t}})$.} Therefore, by applying the Kantorovich duality theorem (\cref{thm:kan_dual}), we get the following duality for all $x_{1:t},y_{1:t}\in\R^{dt}$ and $\mu_{x_{1:t}}, \nu_{y_{1:t}} \in \calP(\R^{d})$,
\begin{equation}
\label{eq:lem_av:6}
\begin{split}
    &~\quad \inf_{\pi_{x_{1:t},y_{1:t}} \in \cpl(\mu_{x_{1:t}},\nu_{y_{1:t}})}\int c(x_{1:t+1},y_{1:t+1})\mathbbm{1}_{\{x_{t+1}\neq y_{t+1}\}} d\pi_{x_{1:t},y_{1:t}}\\
    &= \sup_{\substack{f_{x_{1:t}} \in C_{b}(\R^{d}), g_{y_{1:t}} \in C_{b}(\R^{d})\\ f_{x_{1:t}}(x_{t+1}) + g_{y_{1:t}}(y_{t+1}) \leq c(x_{1:t+1},y_{1:t+1})\mathbbm{1}_{\{x_{t+1} \neq y_{t+1} \}}}} \Big(\int f_{x_{1:t}} d\mu_{x_{1:t}} + \int g_{y_{1:t}} d\nu_{y_{1:t}}\Big),
\end{split}
\end{equation}
{\color{black}and both the primal and dual Kantorovich problems are attainable.} For all $x_{1:t}, y_{1:t} \in \R^{dt}$, $f_{x_{1:t}} \in C_b(\R^{d})$ and $ g_{y_{1:t}} \in C_b(\R^{d})$ such that $f_{x_{1:t}}(x_{t+1}) + g_{y_{1:t}}(y_{t+1}) \leq c(x_{1:t+1},y_{1:t+1})\mathbbm{1}_{\{x_{t+1} \neq y_{t+1} \}}$, we define $f\colon \R^{d(t+1)}\ni (x_{1:t},x_{t+1}) \mapsto f_{x_{1:t}}(x_{t+1}) \in \R$ and $g\colon \R^{d(t+1)}\ni (y_{1:t},y_{t+1}) \mapsto g_{y_{1:t}}(y_{t+1}) \in \R$. Notice that $f$ and $g$ are separately continuous. Then by \cite[Theorem~2.2]{johnson1969separate}, $f$ and $g$ are measurable. Notice that for all $x_{1:t+1}, y_{1:t+1} \in \R^{d(t+1)}$, $f(x_{1:t+1}) + g(y_{1:t+1}) \leq c(x_{1:t+1},y_{1:t+1})\mathbbm{1}_{\{x_{1:t+1} \neq y_{1:t+1} \}}$. By taking $y_{1:t+1} = 0$ and integral over $\pi$, we have $\int |f|d\mu_{1:t+1} \leq 1 + \int \Vert x_{1:t+1}\Vert d\mu_{1:t+1} + \vert g(0) \vert$. Therefore $f \in L^{1}(\mu_{1:t+1})$ and similarly we have $g \in L^{1}(\nu_{1:t+1})$. Let $\eta_{1:t}$ be the marginal distribution of $\pi_{1:t}$ on the diagonal i.e. $\eta_{1:t}(dx_{1:t}) = \int_{\R^{dt}} \mathbbm{1}_{\{x_{1:t}= y_{1:t}\}}\pi_{1:t}(dx_{1:t}, dy_{1:t})$. Then we have
{\color{black}
\begin{equation}
\label{eq:lem_av:7}
    \begin{split}
    &~\quad \int \Big(\int f_{x_{1:t}}(x_{t+1}) \mu_{x_{1:t}}(dx_{t+1}) + \int g_{y_{1:t}}(y_{t+1}) \nu_{y_{1:t}}(dy_{t+1})\Big)\mathbbm{1}_{\{x_{1:t}= y_{1:t}\}}\pi_{1:t}(dx_{1:t},dy_{1:t})\\
    &= \int \Big(\int f(x_{1:t+1}) \mu_{x_{1:t}}(dx_{t+1}) + \int g(y_{1:t+1}) \nu_{y_{1:t}}(dy_{1:t})\Big)\mathbbm{1}_{\{x_{1:t}= y_{1:t}\}}\pi_{1:t}(dx_{1:t},dy_{1:t})\\
    &= \int \int f(x_{1:t+1}) \mu_{x_{1:t}}(dx_{t+1})\eta_{1:t}(dx_{1:t}) + \int \int g(y_{1:t+1}) \nu_{y_{1:t}}(dy_{t+1})\eta_{1:t}(dy_{1:t})\\
    &= \int \int f(x_{1:t+1}) \mu_{x_{1:t}}(dx_{t+1})\mu_{1:t}(dx_{1:t}) +  \int\int g(y_{1:t+1}) \nu_{y_{1:t}}(dy_{t+1})\nu_{1:t}(dy_{1:t})\\
    &\quad - \int\int f(x_{1:t+1}) \mu_{x_{1:t}}(dx_{t+1})(\mu_{1:t} - \eta_{1:t})(dx_{1:t}) - \int\int g(y_{1:t+1}) \nu_{y_{1:t}}(dy_{t+1})(\nu_{1:t} - \eta_{1:t})(dy_{1:t})\\
    &= \int f(x_{1:t+1}) \mu_{1:t+1}(dx_{1:t+1})  +  \int g(y_{1:t+1}) \nu_{1:t+1}(dy_{1:t+1})\\
    &\quad - \int \int \big(f(x_{1:t+1}) + g(y_{1:t+1})\big)(\mu_{x_{1:t}}\otimes\nu_{y_{1:t}})(dx_{t+1},dy_{t+1}) \big((\mu_{1:t} - \eta_{1:t})\otimes(\nu_{1:t} - \eta_{1:t})\big)(dx_{1:t},dy_{1:t})\\
    &\leq \int f(x_{1:t+1}) \mu_{1:t+1}(dx_{1:t+1})  +  \int g(y_{1:t+1}) \nu_{1:t+1}(dy_{1:t+1})\\
    &\quad - \int \int \big\vert f(x_{1:t+1}) + g(y_{1:t+1})\big\vert(\mu_{x_{1:t}}\otimes\nu_{y_{1:t}})(dx_{t+1},dy_{t+1}) \big((\mu_{1:t} - \eta_{1:t})\otimes(\nu_{1:t} - \eta_{1:t})\big)(dx_{1:t},dy_{1:t}).
    \end{split}
\end{equation}
}
We first estimate the $\int f d\mu_{1:t+1}  +  \int g d\nu_{1:t+1}$ term in \eqref{eq:lem_av:7}. By \cref{thm:kan_dual}, we get
\begin{equation}
\label{eq:lem_av:8}
\begin{split}
\int f d\mu_{1:t+1} +  \int g d\nu_{1:t+1} &\leq \sup_{\substack{f \in L^1(\mu_{1:t+1}), g \in L^1(\nu_{1:t+1})\\ f(x_{1:t+1}) + g(y_{1:t+1}) \leq c(x_{1:t+1},y_{1:t+1})\mathbbm{1}_{\{x_{t+1} \neq y_{t+1} \}}}} \Big(\int f d\mu_{1:t+1} +  \int g d\nu_{1:t+1}\Big)\\
&= \inf_{\pi \in \cpl(\mu_{1:t+1}, \nu_{1:t+1})} \int c(x_{1:t+1},y_{1:t+1})\mathbbm{1}_{\{x_{t+1} \neq y_{t+1} \}} d\pi\\
&\leq \inf_{\pi \in \cpl(\mu, \nu)} \int c(x,y)\mathbbm{1}_{\{x\neq y \}} d\pi = \TVD_1(\mu,\nu).
\end{split}
\end{equation}
Next, we estimate the last term in \eqref{eq:lem_av:7}:
\begin{equation}
\label{eq:lem_av:9}
    \begin{split}
    &~\quad \int \int \vert f + g \vert d(\mu_{x_{1:t}}\otimes\nu_{y_{1:t}}) d\big((\mu_{1:t} - \eta_{1:t})\otimes(\nu_{1:t} - \eta_{1:t})\big)\\
    &\leq \int \int c(x_{1:t+1},y_{1:t+1}) d(\mu_{x_{1:t}}\otimes\nu_{y_{1:t}}) d\big((\mu_{1:t} - \eta_{1:t})\otimes(\nu_{1:t} - \eta_{1:t})\big)\\
    &= \int (\int \Vert x_{t+1} \Vert d\mu_{x_{1:t}}  + \int \Vert y_{t+1} \Vert d\nu_{y_{1:t}} + c(x_{1:t},y_{1:t})) d\big((\mu_{1:t} - \eta_{1:t})\otimes(\nu_{1:t} - \eta_{1:t})\big)\\
    &\leq (1+\beta)\int c(x_{1:t},y_{1:t}) d\big((\mu_{1:t} - \eta_{1:t})\otimes(\nu_{1:t} - \eta_{1:t})\big),
    \end{split}
\end{equation}
where the last inequality is by \eqref{eq:lem_av:0}. Also notice that 
\begin{equation}
\label{eq:lem_av:10}
\begin{split}
    \int c(x_{1:t},y_{1:t})\mathbbm{1}_{\{x_{1:t}\neq y_{1:t}\}}d\pi_{1:t} &= \int c(x_{1:t},y_{1:t}) d\pi_{1:t} - \int c(x_{1:t},y_{1:t})\mathbbm{1}_{\{x_{1:t}= y_{1:t}\}} d\pi_{1:t}\\
    &= \int c(x_{1:t},y_{1:t}) d(\mu_{1:t}\otimes \nu_{1:t}) - \int c(x_{1:t},x_{1:t}) d\eta_{1:t}\\
    &= \int c(x_{1:t},y_{1:t}) d\big((\mu_{1:t} - \eta_{1:t})\otimes(\nu_{1:t} - \eta_{1:t})\big),
\end{split}
\end{equation}
where the second equality is due to the separable cost $c(x_{1:t}, y_{1:t})$. Thus by plugging \eqref{eq:lem_av:10} into \eqref{eq:lem_av:9}, we have 
\begin{equation}
\label{eq:lem_av:11}
    \begin{split}
    \int \int \vert f + g \vert d(\mu_{x_{1:t}}\otimes\nu_{y_{1:t}}) d\big((\mu_{1:t} - \eta_{1:t})\otimes(\nu_{1:t} - \eta_{1:t})\big) \leq (1+\beta)\int c(x_{1:t},y_{1:t})\mathbbm{1}_{\{x_{1:t}\neq y_{1:t}\}}d\pi_{1:t}.
    \end{split}
\end{equation}
By combining \eqref{eq:lem_av:7}, \eqref{eq:lem_av:8} and \eqref{eq:lem_av:11}, we have
\begin{equation}
\label{eq:lem_av:12}
\begin{aligned}
    &\quad~\int \Big(\int f_{x_{1:t}} d\mu_{x_{1:t}} + \int g_{y_{1:t}} d\nu_{y_{1:t}}\Big)\mathbbm{1}_{\{x_{1:t}= y_{1:t}\}}d\pi_{1:t}\\
    &\leq (1+\beta)\int c(x_{1:t},y_{1:t})\mathbbm{1}_{\{x_{1:t}\neq y_{1:t}\}}d\pi_{1:t}+ \TVD_1(\mu,\nu).
\end{aligned}
\end{equation}
By combining \eqref{eq:lem_av:4}, \eqref{eq:lem_av:6} and \eqref{eq:lem_av:12}, we have
\begin{equation*}
\begin{split}
    \AVD_1(\mu_{1:t+1},\nu_{1:t+1}) &\leq \inf_{\pi_{1:t} \in \bccpl(\mu_{1:t},\nu_{1:t})} (2+ 2\beta)\int c(x_{1:t},y_{1:t})\mathbbm{1}_{\{x_{1:t}\neq y_{1:t}\}}d\pi_{1:t} + \TVD_1(\mu,\nu)\\
    &= (2+ 2\beta)\AVD_1(\mu_{1:t},\nu_{1:t}) + \TVD_1(\mu,\nu) = (2+ 4\alpha)\AVD_1(\mu_{1:t},\nu_{1:t}) + \TVD_1(\mu,\nu),
\end{split}
\end{equation*}
which completes the proof.
\end{proof}

\begin{lemma}
\label{lem:av_main}
    Let $\alpha > 0$ and $\mu,\nu\in\calP_1(\R^{dT})$ with $\alpha$-linear conditional moments. Then
    \begin{equation}
    \label{eq:lem:av_main:1}
        \AVD_1(\mu, \nu) \leq ((3+4\alpha)^{T} - 1)\TVD_1(\mu,\nu).
    \end{equation}
\end{lemma}
\begin{proof}
    We prove by induction. When $t=1$,
    \begin{equation*}
        \AVD_1(\mu_1,\nu_1) = \TVD_1(\mu_1,\nu_1) \leq \TVD_1(\mu,\nu) \leq ((3 + 4\alpha)^{t} - 1)\TVD_1(\mu,\nu).
    \end{equation*}
    For all $t=1,\dots,T-1$, assume that $\AVD_1(\mu_{1:t},\nu_{1:t}) = ((3 + 4\alpha)^{t} - 1)\TVD_1(\mu,\nu)$. Combining this with Lemma~\ref{lem:av}, then we have
    \begin{equation*}
    \begin{split}
        \AVD_1(\mu_{1:t+1},\nu_{1:t+1}) &\leq (2 + 4\alpha)\AVD_1(\mu_{1:t}, \nu_{1:t}) + \TVD_1(\mu,\nu)\\
        &\leq (2 + 4\alpha)((3 + 4\alpha)^{t} - 1)\TVD_1(\mu,\nu) + \TVD_1(\mu,\nu)\\
        &= \Big[(2 + 4\alpha)(3 + 4\alpha)^{t} - (2 + 4\alpha) + 1\Big]\TVD_1(\mu,\nu) \leq  (3 + 4\alpha)^{t+1}\TVD_1(\mu,\nu).
    \end{split}
    \end{equation*}
    By induction, we obtain \eqref{eq:lem:av_main:1} and complete the proof.
\end{proof}

\begin{theorem}[Metric dominations]\label{thm:metric_dom_new}
Let $\alpha > 0$ and $\mu,\nu\in\calP_1(\R^{dT})$ with $\alpha$-linear conditional moments. Then 
\begin{equation*}
    \AWD_{1}(\mu,\nu) \leq ((3 + 4\alpha)^{T} - 1)\TVD_{1}(\mu,\nu).
\end{equation*}
\end{theorem}
\begin{proof}[Proof of Theorem~\ref{thm:metric_dom_new}]
By definitions of $\AVD_1(\cdot,\cdot)$ and $\AWD_1(\cdot,\cdot)$, we have 
\begin{equation*}
\begin{split}
    \AWD_1(\mu,\nu) &= \inf_{\pi^{\mathrm{bc}} \in \bccpl(\mu,\nu)}\int \Vert x - y \Vert \mathbbm{1}_{\{x\neq y\}}\pi(dx,dy)\\
    &\leq \inf_{\pi^{\mathrm{bc}} \in \bccpl(\mu,\nu)}\int (\Vert x \Vert + \Vert y \Vert + 1)\mathbbm{1}_{\{x\neq y\}}\pi(dx,dy) = \AVD_1(\mu,\nu).
\end{split}
\end{equation*}
Thus, by Lemma~\ref{lem:av_main} we conclude that $\AWD_1(\mu,\nu) \leq \AVD_1(\mu,\nu)\leq ((3 + 4\alpha)^{T} - 1)\TVD_{1}(\mu,\nu)$.
\end{proof}
The linear conditional moments condition in \cref{thm:metric_dom_new} can not be relaxed to moment conditions, e.g. see the counterexample below.
\begin{example}
\label{ex:metric_domination}
    For all $\epsilon \in (0,1)$, let $\mu^{\epsilon} = \epsilon(1-\epsilon)\delta_{(1,\frac{1}{\epsilon})} + \epsilon^{2}\delta_{(1,0)} + (1-\epsilon)\delta_{(0,0)}$ and $\nu^{\epsilon} = \epsilon(1-\epsilon)\delta_{(1,\frac{1}{\epsilon})} + (1-\epsilon + \epsilon^2)\delta_{(0,0)}$; see Figure~\ref{fig:ex_mu_nu_epsilon} for visualization. Note that for all $\epsilon \in (0,1)$, $M_1(\mu^{\epsilon}) \leq 2$. and $M_1(\nu^{\epsilon}) \leq 2$. {\color{black}However, with $x_1 = 1$, we get $\int \Vert x_{2} \Vert d\nu^\epsilon_{x_{1}} = \frac{1}{\epsilon} \to \infty$ as $\epsilon \to 0$, which implies that $\mu$ fails to admit linear conditional moments.} Then we compute that $\TVD_1(\mu^\epsilon, \nu^\epsilon) = 2\epsilon^2$ and $\AVD_1(\mu^\epsilon, \nu^\epsilon) = 2 + \epsilon - \epsilon^2$. Thus we have
    \begin{equation*}
        \lim_{\epsilon \to 0}\frac{\AVD_1(\mu^\epsilon, \nu^\epsilon)}{\TVD_1(\mu^\epsilon, \nu^\epsilon)} = \lim_{\epsilon \to 0} \frac{2 + \epsilon - \epsilon^2}{2\epsilon^2} = +\infty,
    \end{equation*}
    which implies that there is no uniform Lipschitz constant depending only on the first moments of measures in $\calP_1(\R^{dT})$ but not depending on conditional moments, such that \cref{lem:av_main} holds.
    \begin{figure}[H]
    \centering
    \begin{subfigure}{.4\textwidth}
      \centering
      \includegraphics[width=0.8\linewidth]{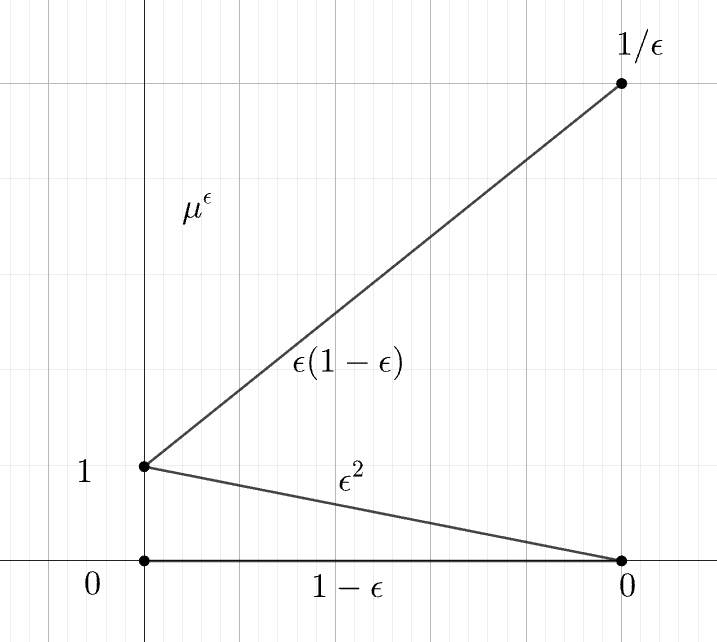}
      \caption{$\mu^\epsilon$}
    \end{subfigure}
    \begin{subfigure}{.4\textwidth}
      \centering
      \includegraphics[width=0.8\linewidth]{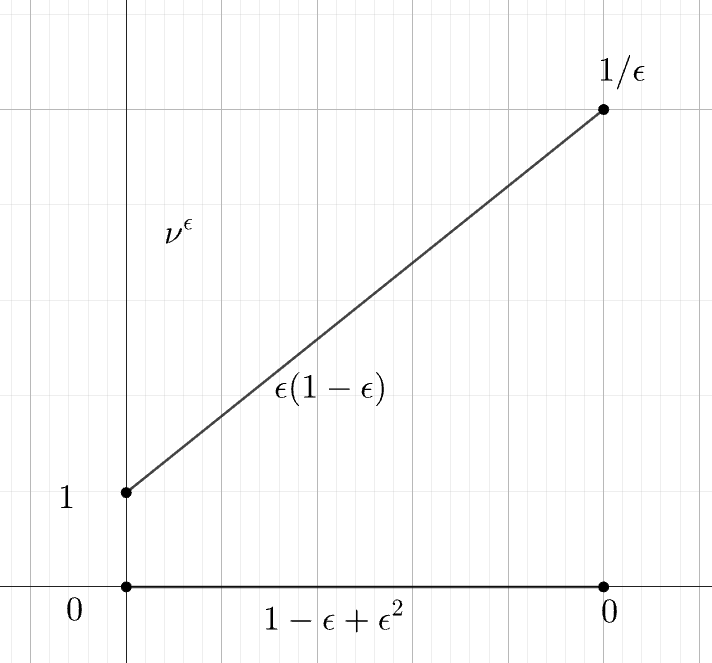}
      \caption{$\nu^\epsilon$}
    \end{subfigure}
    \caption{Visualization of $\mu^\epsilon$ and $\nu^\epsilon$.}
    \label{fig:ex_mu_nu_epsilon}
\end{figure}
\end{example}

\section{Smooth distances}\label{sec:smooth_distance}
In this section, we fix the bandwidth $\sigma$ and analyze the error 
between smoothed empirical measures $\mu^N \ast \calN_{\sigma}$ and the smoothed underlying measure $\mu_{\sigma} = \mu \ast \calN_{\sigma}$. For simplicity, we refer to the distance between two smoothed measures as the smooth distance between the measures. For all $\sigma > 0$, we denote by $\AWD_{1}^{(\sigma)}(\mu,\nu) = \AWD_{1}(\mu \ast \calN_{\sigma} , \nu \ast \calN_{\sigma})$ and $\TVD_{1}^{(\sigma)}(\mu,\nu) = \TVD_{1}(\mu \ast \calN_{\sigma} , \nu \ast \calN_{\sigma})$.

\subsection{Convergence under smooth $\TVD_1$}
\label{sec:smooth_distance_tvd}
\begin{theorem}[Mean convergence under smooth $\TVD_1$]\label{thm:smooth_tv_mean}
Let $p >2$ and $\mu \in \calP_1(\R^{dT})$ with finite $p$-th moment. Then there exist $C_1, C_2 > 0$ such that, for all $N \in \N$,
\begin{equation}
\label{eq:thm:smooth_tv_mean:0}
    \E\Big[\TVD_{1}^{(\sigma)}(\mu, \mu^N) \Big] \leq C_{\sigma, p,M_p}N^{-\frac{1}{2}},
\end{equation}
where $M_p = M_p(\mu)$ and
\begin{equation}\label{eq:C}
    C_{\sigma,p,M_p} = \left(\int \frac{(\Vert x \Vert + \frac{1}{2})^2}{1 + \Vert x \Vert^p}dx\right)^{\frac{1}{2}}  \sqrt{ \big(C_1(2^p M_p + 1) + C_2 2^{p} \sigma^{p}\big)\frac{1}{(2\pi\sigma)^{dT}}}.
\end{equation}
\end{theorem}
\begin{proof}[Proof of \cref{thm:smooth_tv_mean}]
Recall that we denote the density of the Gaussian kernel by $\varphi_{\sigma}$. Since $\varphi_{\sigma}$ is smooth, by convolution, $\mu \ast \calN_{\sigma}$ and $\mu^N \ast \calN_{\sigma}$ also have smooth densities, and we denote them by $q$ and $q^N$. Let $f_p\colon \R^{dT} \to \R_{\geq 0}$ s.t. $f_p(x) = \frac{1}{1 + \Vert x \Vert^{p}}$ for all $x\in\R^{dT}$. By Cauchy-Schwarz, we have 
\begin{equation}\label{eq:thm:smooth_tv_mean:1}
\begin{split}
    \E\Big[\TVD_{1}^{(\sigma)}(\mu , \mu^N ) \Big] &= \E\Big[\int (\Vert x \Vert + \frac{1}{2}) |q - q_N|(x)dx \Big] \\
    &\leq \left(\int (\Vert x \Vert + \frac{1}{2})^2 f_p(x)dx\right)^{\frac{1}{2}}\left(\int\frac{\E[(q(x) - q_N(x))^2]}{f_p(x)}dx\right)^{\frac{1}{2}}.
\end{split}
\end{equation}
Notice that $\E[q_N(x)] = q(x)$. We have
\begin{equation*}
    \begin{split}
        \E\Big[(q(x) - q_N(x))^2\Big] &= \Var\left[q_N(x)\right] = \Var\Big[\frac{1}{N}\sum_{i=1}^N \varphi_{\sigma}(x - X^{(i)})\Big]
        = \frac{1}{N}\Var\Big[\varphi_{\sigma}(x - X^{(1)})\Big]\\
        &\leq \frac{1}{N}\E\Big[\varphi^2_{\sigma}(x - X^{(1)})\Big] = \frac{1}{N}\frac{1}{(2\pi\sigma^2)^{dT}}\E\Big[ e^{-\frac{\Vert x - X^{(1)} \Vert^2}{\sigma^2}}\Big].
    \end{split}
\end{equation*}
This implies that 
\begin{equation}
\label{eq:thm:smooth_tv_mean:2}
        \int\frac{\E\Big[(q(x) - q_N(x))^2\Big]}{f_p(x)}dx \leq \frac{1}{(2\pi\sigma^2)^{dT}}\cdot\frac{1}{N}\E\Big[\int \frac{1}{f_p(x)}\cdot e^{-\frac{\Vert x - X^{(1)} \Vert^2}{\sigma^2}} dx\Big].
\end{equation}
Notice that
\begin{equation*}
\begin{split}
    &\quad\E\Big[\int \frac{1}{f_p(x)}\cdot e^{-\frac{\Vert x - X \Vert^2}{\sigma^2}} dx\Big] = \int\int \frac{1}{f_p(x)}\cdot e^{-\frac{\Vert x - z \Vert^2}{\sigma^2}} dx\mu(dz)\\
    &= \int\int \frac{1}{f_p(z+y)}\cdot e^{-\frac{\Vert y \Vert^2}{\sigma^2}} dy\mu(dz),\quad (x = y+z)\\
    &= \int\int (\Vert z+y\Vert^{p} + 1)\cdot e^{-\frac{\Vert y \Vert^2}{\sigma^2}} dy\mu(dz)
    \leq \int\int \left(2^p\left(\Vert z\Vert^p +\Vert y\Vert^{p}\right)+1\right)\cdot e^{-\frac{\Vert y \Vert^2}{\sigma^2}} dy\mu(dz)\\
    &= \int  (2^p \Vert z\Vert^p + 1) \int e^{-\frac{\Vert y \Vert^2}{\sigma^2}} dy\mu(dz) + \int2^p \Vert y\Vert^{p} \cdot e^{-\frac{\Vert y \Vert^2}{\sigma^2}} dy\\
    &= \int (2^p \Vert z\Vert^p + 1) \int \sigma^{dT} e^{-\Vert u \Vert^2} du\mu(dz) + \int2^p \sigma^{dT+p}\Vert u\Vert^{p} \cdot e^{-\Vert u \Vert^2} du,\quad (y = \sigma u)\\
    &= C_1 \sigma^{dT}\int (2^p\Vert z\Vert^p + 1) \mu(dz) + C_2 2^p \sigma^{dT+p} = \sigma^{dT}\big(C_1(2^p M_p + 1) + C_2 2^{p} \sigma^{p}\big).
\end{split}
\end{equation*}
where $C_1$ and $C_2$ are appropriate constants.
Therefore, by combining this, \eqref{eq:thm:smooth_tv_mean:1} and \eqref{eq:thm:smooth_tv_mean:2}, we obtain that 
\begin{equation*}
    \E\Big[\TVD_{1}^{(\sigma)}(\mu , \mu^N ) \Big] \leq \left(\int \frac{(\Vert x \Vert + \frac{1}{2})^2}{1 + \Vert x \Vert^p}dx\right)^{\frac{1}{2}}  \sqrt{ \big(C_1(2^p M_p + 1) + C_2 2^{p} \sigma^{p}\big)\frac{1}{(2\pi\sigma)^{dT}}\frac{1}{N}}.
\end{equation*}
Therefore, by setting $C_{\sigma,p,M_p}$ as \eqref{eq:C}, we prove \eqref{eq:thm:smooth_tv_mean:0}.
\end{proof}

\begin{remark}
\label{rem:subgaussian_kernel}
    \cref{thm:smooth_tv_mean} holds  not only for Gaussian kernel $\calN_\sigma$, but also for a broad class of sub-Gaussian kernels. Let $\calG_{\sigma} \in \calP(\R^{dT})$ with density $g_{\sigma}$ that decomposes as $g_{\sigma}(x) = \prod_{j=1}^{dT}\tilde{g}_{\sigma}(x_j)$ and the measure with density $\tilde{g}_{\sigma}$ is $\sigma$-subgaussian, bounded and monotonically decreasing as its argument goes away from zero in either direction. Let $\delta = \min\{1, \frac{1}{4\sigma^2}\}$, then by Lemma~2 in \cite{Goldfeld2020GaussianSmoothedOT}, there exists a constant $c_1 > 0$ such that for all $x \in \R^{dT}$, $g_{\sigma}(x) \leq c_1^{dT}e^{\delta\Vert x \Vert^2}\varphi_{\sigma}(x)$. Then by replacing $\calN_{\sigma}$ with $\calG_{\sigma}$, \cref{thm:smooth_tv_mean} still holds but with a different constant. For details, see \cite{Goldfeld2020GaussianSmoothedOT}.
\end{remark}

\begin{theorem}[Deviation convergence under smooth $\TVD_1$]
\label{thm:smooth_tv_deviation}
Let $K \subseteq \R^{dT}$ be compact and $\mu \in \calP(K)$. Then there exists $c_1 > 0$ s.t. for all $x > 0$ and $N \in \N$,
\begin{equation}
\label{eq:thm:smooth_tv_deviation:0}
    \sP\Big(\TVD_{1}^{(\sigma)}(\mu , \mu^N ) - \E\!\left[\TVD_{1}^{(\sigma)}(\mu , \mu^N )\right] \geq x \Big) \leq e^{-\frac{Nx^2}{c_{\sigma,K}^2}},
\end{equation}
where
\begin{equation}\label{eq:c}
    c_{\sigma,K} = c_1 (\sup_{x\in K}\frac{1 + 2\Vert x \Vert}{2\sigma} + 1).   
\end{equation}
\end{theorem}

\begin{proof}[Proof of \cref{thm:smooth_tv_deviation}]
In the proof, we apply McDiarmid's inequality; see \cite{mcdiarmid1989method}, to $\TVD_{1}^{(\sigma)}(\mu , \mu^N )$. First, we derive a variational expression of $\TVD_{1}^{(\sigma)}(\mu , \mu^N ) = \TVD_{1}(\mu \ast \calN_{\sigma} , \mu^N \ast \calN_{\sigma})$. Let $\calF = \{f \in \calB(\R^d,\R) \colon |f(x)| \leq (\Vert x \Vert + \frac{1}{2}), \forall x \in \R^{dT}\}$, $\calF_{\sigma} = \{f\ast\varphi_{\sigma} \colon f \in \calF\}$. 
Since $\varphi_{\sigma}$ is smooth, then by convolution, $\mu \ast \calN_{\sigma}$ and $\mu^N \ast \calN_{\sigma}$ also have smooth densities, and we denote them by $q$ and $q^N$. Let $f^{*}(x) = \mathrm{sign}(q^N(x) -q(x))\cdot (\Vert x \Vert + \frac{1}{2}) \in \calF$. Then, we have
\begin{equation}
\label{eq:thm:smooth_tv_deviation:1}
\begin{split}
    \TVD_{1}^{(\sigma)}(\mu , \mu^N ) 
    &= \int(\Vert x \Vert + \frac{1}{2}) |q(x) - q^N(x)|dx
    = \int f^{*}(x) q^N(x)dx - \int f^{*}(x) q(x)dx \\
    &= \sup_{f\in\calF}\left(\int_{\R^{dT}}  f(x) \left(\int_{\R^{dT}}\varphi_{\sigma}(x-y)\mu^N(dy) - \int_{\R^{dT}}\varphi_{\sigma}(x-y)\mu(dy)\right)dx\right)\\
    &=\sup_{f\in\calF}\left(\int_{\R^{dT}}(  f \ast\varphi_{\sigma})(y)\mu^N(dy) - \int_{\R^{dT}}(f\ast\varphi_{\sigma})(y)\mu(dy)\right)\\
    &=\sup_{f\in\calF}\left(\frac{1}{N}\sum_{n=1}^{N}(f \ast\varphi_ {\sigma})(X^{(n)}) - \int_{\R^{dT}}(f\ast\varphi_{\sigma})(y)\mu(dy)\right).
\end{split}
\end{equation}
Let $F \colon K^N \to \R$ s.t. for all $(x_1, \dots, x_N) \in K^N$,
\begin{equation*}
    F(x_1,\dots,x_N) = \sup_{f\in\calF}\left(\frac{1}{N}\sum_{n=1}^{N}(f \ast\varphi_{\sigma})(x_n) - \int_{\R^{dT}}(f\ast\varphi_{\sigma})(y)\mu(dy)\right).
\end{equation*}
Next, we show that $F$ satisfies the conditions to apply the McDiarmid's inequality. For all $(x_1, \dots, x_N)$, $(x_1^{\prime},\dots, x_N^{\prime}) \in K^N$ that differ only in the $i$-th coordinate, $i =1,\dots,N$, we have that 
\begin{equation}
\label{eq:thm:smooth_tv_deviation:2}
    \begin{split}
        &\quad F(x_1, \dots, x_N) - F(x'_1, \dots, x'_N)\\
        &= \sup_{f\in\calF}\left(\frac{1}{N}\sum_{n=1}^{N}(f \ast\varphi_{\sigma})(x_n) - \int_{\R^{dT}}(f\ast\varphi_{\sigma})(y)\mu(dy)\right)\\
        &\qquad - \sup_{g\in\calF}\left(\frac{1}{N}\sum_{n=1}^{N}(g \ast\varphi_{\sigma})(x'_n) - \int_{\R^{dT}}(g\ast\varphi_{\sigma})(y)\mu(dy)\right)\\
        &\leq \sup_{f\in\calF}\left(\frac{1}{N}\sum_{n=1}^{N}(f \ast\varphi_{\sigma})(x_n) - \frac{1}{N}\sum_{n=1}^{N}(f \ast\varphi_{\sigma})(x'_n)\right) = \sup_{f\in\calF}\left(\frac{1}{N}(f \ast\varphi_{\sigma})(x_i) - \frac{1}{N}(f \ast\varphi_{\sigma})(x'_i)\right).
    \end{split}
\end{equation}
Notice that for all $f \in \calF$, $x = (x^{(1)},\dots,x^{(dT)})\in K$ and $j=1,\cdots,dT$,
\begin{equation*}
\begin{split}
    &\quad~\frac{\partial }{\partial x^{(j)}}(f \ast\varphi_{\sigma})(x) \\
    &= \frac{\partial }{\partial x^{(j)}}\int_{\R^{dT}}\varphi_{\sigma}(x-y)f(y)dy
    = \int_{\R^{dT}}\frac{\partial }{\partial x^{(j)}}\varphi_{\sigma}(x-y)f(y)dy\\
    &= -\int_{\R^{dT}} \frac{y^{(j)} - x^{(j)}}{\sigma^2}\varphi_{\sigma}(y-x)f(y)dy
    = -\int_{\R^{dT}} \frac{z^{(j)}}{\sigma^2}\varphi_{\sigma}(z)f(x+z)dz,\quad (z = y-x)\\
    &\leq \int_{\R^{dT}} \frac{\vert z^{(j)} \vert}{\sigma^2}\varphi_{\sigma}(z)(\frac{1}{2} + \Vert x + z \Vert) dz = \sigma^{dT}\int_{\R^{dT}} \vert \frac{u^{(j)}}{\sigma} \vert\varphi_{1}(u)(\frac{1}{2} + \Vert x + \sigma u \Vert) du,\quad (z = \sigma u)\\
    &\leq \sigma^{dT}\int_{\R^{dT}} \frac{\frac{1}{2} + \Vert x \Vert}{\sigma} \vert u^{(j)}\vert \varphi_{1}(u) du + \sigma^{dT}\int_{\R^{dT}} \vert u^{(j)} \vert\varphi_{1}(u)\Vert u \Vert du\\
    &\leq \sigma^{dT}(\frac{1 + 2\Vert x \Vert}{2\sigma} + 1)\int_{\R^{dT}} \vert u^{(j)} \vert (1 + \Vert u \Vert)\varphi_{1}(u) du
    \leq \sigma^{dT}(\sup_{x\in K}\frac{1 + 2\Vert x \Vert}{2\sigma} + 1)\int_{\R^{dT}} \vert u^{(j)} \vert (1 + \Vert u \Vert)\varphi_{1}(u) du.
\end{split}
\end{equation*}
Let $c_1 = \sigma^{dT}\int_{\R^{dT}} \vert u^{(1)} \vert (1 + \Vert u \Vert)\varphi_{1}(u) du$ and $c_{\sigma,K} = c_1(\sup_{x\in K}\frac{1 + 2\Vert x \Vert}{2\sigma} + 1)$. Thus for all $x,x' \in K$, we have 
\begin{equation*}
    \sup_{f\in\calF}|(f \ast\varphi_{\sigma})(x) -  (f \ast\varphi_{\sigma})(x')| \leq c_{\sigma,K} \Vert x - x'\Vert.
\end{equation*}
Combine this with \eqref{eq:thm:smooth_tv_deviation:2}. We have for all $(x_1, \dots, x_N)$, $(x_1^{\prime},\dots, x_N^{\prime}) \in K^N$ that differ only in the $i$-th coordinate, $i =1,\dots,N$, $F(x_1, \dots, x_N) - F(x'_1, \dots, x'_N) \leq \frac{c_{\sigma,K}}{N}\Vert x_i - x'_i\Vert$. Therefore, we can apply McDiarmid's inequality; see \cite{mcdiarmid1989method}, to conclude that for all $x >0$, $N \in \N$, 
\begin{equation*}
    \sP\Big( F(X^{(1)},\dots,X^{(N)}) - \E[F(X^{(1)},\dots,X^{(N)})] \geq x \Big) \leq e^{-\frac{Nx^2}{c_{\sigma, K}^2}}.
\end{equation*}
Combining this, the definition of $F$ and \eqref{eq:thm:smooth_tv_deviation:1}, we prove \eqref{eq:thm:smooth_tv_deviation:0}.
\end{proof}

\subsection{Convergence under smooth $\AWD_1$}
\label{sec:smooth_distance_awd}
In this subsection, we extend the convergence results under smooth $\TVD_1$ to smooth $\AWD_1$ with \cref{thm:metric_dom_new}.

\begin{theorem}[Convergence rates under smooth $\AWD_1$]\label{thm:smooth_adw_mean_deviation}
Let $K \subseteq \R^{dT}$ be compact and $\mu \in \calP(K)$. Then for all $\sigma \in (0,1]$, there exist $C_K, C_{\sigma,p,M_p}, c_{\sigma,K} > 0$ such that, for all $x > 0$ and $N \in \N$,
\begin{equation}
\label{eq:thm:smooth_adw_mean_deviation:1}
    \E\left[\AWD_{1}^{(\sigma)}(\mu , \mu^N ) \right] \leq C_K C_{\sigma, p,M_p}N^{-\frac{1}{2}},
\end{equation}
\begin{equation}
\label{eq:thm:smooth_adw_mean_deviation:2}
    \sP\!\left(\AWD_{1}^{(\sigma)}(\mu , \mu^N ) \geq x + C_K C_{\sigma, p,M_p}N^{-\frac{1}{2}} \right) \leq e^{-\frac{Nx^2}{C_K^2 c_{\sigma, K}^2}},
\end{equation}
where $C_{\sigma,p,M_p}$ is given by \eqref{eq:C} in \cref{thm:smooth_tv_mean} and $c_{\sigma,K}$ is given by \eqref{eq:c} in \cref{thm:smooth_tv_deviation}.
\end{theorem}

\begin{proof}[Proof of \cref{thm:smooth_adw_mean_deviation}]
Since $\mu$ and $\mu^N$ are supported on $K$, then by {\color{black}Lemma~\ref{lem:cond_kernel}-(ii)}, for all $\sigma\in (0,1]$, $\mu\ast \calN_{\sigma}$ and $\mu^N\ast \calN_{\sigma}$ have $\alpha_K$-linear conditional moments, where $\alpha_K = \sup_{x\in K}\Vert x \Vert + M_1(\calN_{1,d})$. Let $C_K = (3+4\alpha_K)^{T}-1$. Then by \cref{thm:metric_dom_new}, for all $\sigma \in (0,1]$, $\AWD_{1}^{(\sigma)}(\mu , \mu^N ) \leq C_K \TVD_{1}^{(\sigma)}(\mu , \mu^N )$. By combining this and \cref{thm:smooth_tv_mean}, we prove \eqref{eq:thm:smooth_adw_mean_deviation:1}. By combining this, \cref{thm:smooth_tv_mean} and \cref{thm:smooth_tv_deviation}, we conclude that
\begin{equation*}
\begin{split}
    \sP\!\left(\AWD_{1}^{(\sigma)}(\mu , \mu^N ) \geq x + C_KC_{\sigma, p,M_p}N^{-\frac{1}{2}} \right) &\leq \sP\!\left(\TVD_{1}^{(\sigma)}(\mu , \mu^N ) \geq \frac{x}{C_K} + C_{\sigma, p,M_p}N^{-\frac{1}{2}} \right)\\
    &\leq \sP\!\left(\TVD_{1}^{(\sigma)}(\mu , \mu^N ) - \E\!\left[\TVD_{1}^{(\sigma)}(\mu , \mu^N )\right] \geq \frac{x}{C_K} \right) \leq e^{-\frac{Nx^2}{C_K^2 c_{\sigma, K}^2}},
\end{split}
\end{equation*}
which proves \eqref{eq:thm:smooth_adw_mean_deviation:2}.
\end{proof}
\begin{remark}
    Although we assume $\mu$ compactly supported in \cref{thm:smooth_adw_mean_deviation}, bounding $\AWD$-distance with $\TVD$-distance for compactly supported measures \cite[Lemma~3.5]{eckstein2024computational} is not enough for the proof of \cref{thm:smooth_adw_mean_deviation}, because we need apply the metric domination theorem to $\mu\ast \calN_\sigma$ and $\mu^N\ast \calN_\sigma$ which are unbounded.
\end{remark}
\begin{lemma}\label{lem:smooth_adw_as_compact}
Let $K \subseteq \R^{dT}$ be compact and $\mu \in \calP(K)$. Then for all $\sigma \in (0,1]$, $\lim_{N \to \infty}\AWD_{1}^{(\sigma)}(\mu , \mu^N ) = 0$, $\sP$-a.s.
\end{lemma}
\begin{proof}
By setting $x = N^{-\frac{1}{4}}$ in \cref{thm:smooth_adw_mean_deviation}, there exist $C, c_{\sigma, K}>0$ such that for all $x > 0$ and $N \in \N$, $\sP\!\left(\AWD_{1}^{(\sigma)}(\mu , \mu^N ) \geq N^{-\frac{1}{4}} + CN^{-\frac{1}{2}} \right) \leq e^{-\frac{N^{\frac{1}{2}}}{C_K^2 c_{\sigma, K}^2}}$. Notice that $\lim_{N \to 0} N^{-\frac{1}{4}} + CN^{-\frac{1}{2}} = 0 $ and $\sum_{N \to \infty} e^{-\frac{N^{\frac{1}{2}}}{C_K^2 c_{\sigma, K}^2}} < \infty$. Thus, by Borel-Cantelli Lemma, we complete the proof.
\end{proof}
\begin{theorem}[Almost sure convergence under smooth $\AWD_1$]\label{thm:smooth_adw_as}
Let $\mu \in \calP_1(\R^{dT})$. Then for all $\sigma \in (0,1]$, $\lim_{N \to \infty}\AWD_{1}^{(\sigma)}(\mu , \mu^N ) = 0$, $\sP$-a.s.
\end{theorem}
\begin{proof}[Proof of \cref{thm:smooth_adw_as}]
    The idea of the proof is to construct a measure $\nu \in \calP(\R^{dT})$ that is compactly supported to apply \cref{lem:smooth_adw_as_compact}, but still very close to $\mu$ under the adapted Wasserstein distance. By \cref{lem:compact_approx_2}, for all $\epsilon>0$, there exists $\nu$ compactly supported s.t. 
    \begin{equation}\label{eq:thm:smooth_adw_as:1}
    \begin{split}
         \AWD_{1}^{(\sigma)}(\mu  , \nu  ) \leq \epsilon\quad \text{and}\quad \lim_{N\to\infty}\AWD_1^{(\sigma)}(\mu^N, \nu^N) \leq \epsilon,\quad \sP\text{-a.s.}
    \end{split}
    \end{equation}
    Since $\nu$ is compactly supported, by \cref{lem:smooth_adw_as_compact}, we have 
    $\lim_{N \to \infty} \AWD_1^{(\sigma)}(\nu, \nu^N) = 0, \quad \sP\text{-a.s.}$ By combining \eqref{eq:thm:smooth_adw_as:1}, this, and triangle inequality, we conclude that $\lim_{N\to\infty}\AWD_{1}^{(\sigma)}(\mu ,\mu^N) \leq 2 \epsilon$. By arbitrarity of $\epsilon$, we complete the proof.
\end{proof}

\section{Bandwidth effect}
\label{sec:bw_effect}
In this section, we focus on the bandwidth effect, namely the convergence of $\AWD_1(\mu, \mu \ast \mathcal{N}_{\sigma})$ as $\sigma$ approaches zero. We denote by $\mu_\sigma = \mu \ast \calN_\sigma$.
\subsection{Lipschitz kernels}
\label{sec:bw_effect_lip}
\begin{theorem}[Lipschitz stability]\label{thm:bw_lip_limit}
Let $L>0$ and $\mu \in \calP_1(\R^{dT})$ with $L$-Lipschitz kernels. Then there exists $\tilde{C}_{L} > 0$ s.t. for all $\sigma > 0$, $\AWD_1(\mu,\mu\ast \calN_\sigma) \leq \tilde{C}_{L} \sigma$.
\end{theorem}
\begin{proof}[Proof of \cref{thm:bw_lip_limit}]
    Recall the proof of Lemma~3.1 in \cite{backhoff2022estimating}, which does not depend on the compactness of $\mu$. Lemma~3.1 in \cite{backhoff2022estimating} states that there exists $C_L >0$ s.t. for all $\sigma > 0$,
    \begin{equation}
    \label{eq:thm:bw_lip_limit:1}
        \AWD_{1}(\mu,\mu_\sigma) \leq C_L\WD_{1}(\mu_1,(\mu_\sigma)_1) + C_L\sum_{t=1}^{T-1}\int\WD_{1}(\mu_{x_{1:t}},(\mu_\sigma)_{x_{1:t}})\mu_\sigma(dx_{1:t}).
    \end{equation}
    For the first term in \eqref{eq:thm:bw_lip_limit:1},
    \begin{equation}
    \label{eq:thm:bw_lip_limit:2}
        \WD_{1}(\mu_1,(\mu_\sigma)_1) \leq \sigma \int_{\sR^d} \Vert x_1\Vert \calN_{1,d}(dx_1) = \sigma M_1(\calN_{1,d}).
    \end{equation}
    Thus we remain to estimate the second term in \eqref{eq:thm:bw_lip_limit:1}. {\color{black} By \cref{lem:cond_kernel}-(iv), we have for all $t = 1,\dots,T-1$, $x_{1:t} \in \R^{dt}$}, 
    \begin{equation}
    \label{eq:thm:bw_lip_limit:5}
        \int\WD_{1}(\mu_{x_{1:t}},(\mu_\sigma)_{x_{1:t}})\mu_\sigma(dx_{1:t}) \leq \sigma \Big(M_1(\calN_{1,d}) + L M_1(\calN_{1,dt})\Big).
    \end{equation}   
    Finally, by combining \eqref{eq:thm:bw_lip_limit:1}, \eqref{eq:thm:bw_lip_limit:2} and \eqref{eq:thm:bw_lip_limit:5}, we conclude that 
    \begin{equation*}
    \begin{split}
        \AWD_{1}(\mu,\mu\ast\calN_{\sigma}) &\leq C_{L} M_1(\calN_{1,d}) \sigma + C_{L}\sum_{t=1}^{T-1}\Big(M_1(\calN_{1,d}) + L M_1(\calN_{1,dt})\Big) \sigma \leq \tilde{C}_L \sigma,
    \end{split}  
    \end{equation*}
    where $\tilde{C}_L = C_{L} M_1(\calN_{1,d}) + C_{L}\sum_{t=1}^{T-1}\big(M_1(\calN_{1,d}) + L M_1(\calN_{1,dt})\big)$. This completes the proof.
\end{proof}

\subsection{Measurable kernels}
First, we relax the Lipschitz kernels assumption in \cref{thm:bw_lip_limit} to continuous kernels.
\begin{definition}
\label{def:cts_kernel}
    We say that $\mu\in\mathcal{P}_1(\R^{dT})$ has continuous kernels if there exists an integration of $\mu$ s.t. for all $t=1,\dots,T-1$, $x_{1:t} \mapsto \mu_{x_{1:t}}$ is continuous ($\calP(\R^{d})$ equipped with $\WD_1$).
\end{definition}
\begin{lemma}\label{lem:bw_cts_limit_compact}
Let $K \subseteq \R^{dT}$ be compact and $\mu \in \calP(K)$ with continuous kernels. Then for all $\epsilon > 0$, there exists $\sigma_{\epsilon} > 0$ s.t. for all $\sigma < \sigma_{\epsilon}$, $\AWD_1(\mu,\mu_\sigma) \leq \epsilon$.
\end{lemma}
\begin{proof}
Lemma~5.1. in \cite{backhoff2022estimating} states that for all $\epsilon > 0$ there exists $C_\epsilon > 0$ s.t. for all $\sigma > 0$,
\begin{equation}
\label{eq:lem:bw_cts_limit_compact:1}
    \AWD_{1}(\mu,\mu_\sigma) \leq \epsilon + C_{\epsilon}\WD_{1}(\mu_1,(\mu_\sigma)_1) + C_{\epsilon}\sum_{t=1}^{T-1}\int\WD_{1}(\mu_{x_{1:t}},(\mu_\sigma)_{x_{1:t}})\mu_\sigma(dx_{1:t}).
\end{equation}
{\color{black} By \cref{lem:cond_kernel}-(v), there exists $\sigma_\epsilon \in (0, \frac{\epsilon}{C_\epsilon})$ s.t. for all $\sigma \in (0,\sigma_\epsilon)$, $t=1,\dots,T-1$, $x_{1:t} \in \R^{dt}$, 
    \begin{equation}
    \label{eq:lem:bw_cts_limit_compact:2}
        \int_{\R^{dt}} \WD_{1}(\mu_{x_{1:t}},(\mu_\sigma)_{x_{1:t}})\mu_\sigma(dx_{1:t}) \leq \frac{\epsilon}{C_\epsilon}.
    \end{equation}      
Combine \eqref{eq:lem:bw_cts_limit_compact:1} and \eqref{eq:lem:bw_cts_limit_compact:2}. We get
\begin{equation*}
\begin{aligned}
    \AWD_{1}(\mu,\mu_\sigma) \leq \epsilon + C_{\epsilon}\sigma + (T-1)\epsilon \leq (T+1) \epsilon.
\end{aligned}
\end{equation*}
Then by re-scaling $\epsilon$, we complete the proof.
} 
\end{proof}
Next, we relax the continuous kernels assumption in \cref{lem:bw_cts_limit_compact} to measurable kernels by Lusin’s theorem and Tietze’s extension theorem.
\begin{lemma}\label{lem:bw_limit_compact}
Let $K \subseteq \R^{dT}$ be compact and $\mu \in \calP(K)$. Then for all $\epsilon > 0$, there exists $\sigma_{\epsilon} > 0$ s.t. for all $\sigma < \sigma_{\epsilon}$, $\AWD_1(\mu,\mu_\sigma) \leq \epsilon$.
\end{lemma}
\begin{proof}
We follow the same idea in proving Theorem~1.3 in \cite{backhoff2022estimating}. We provide the proof for a two-period setting, that is $T = 2$. The general case follows by the same arguments applying Lusin's theorem recursively at each time, however it involves a lengthy backward induction. W.l.o.g. we let $K = [0,1]^{dT}$ be the unit closed cube on $\R^{dT}$. Let $\epsilon>0$ and we would like to construct $\nu \in \calP(\R^{dT})$ s.t. $\nu$ has continuous kernels and $\AWD_1(\mu,\nu)\leq T\sqrt{d}\epsilon$. First, by Lusin’s theorem there exists a compact set $\tilde{K} \subseteq [0,1]^{d}$ such that $\mu([0,1]^{d}\backslash\tilde{K}) > 1-\epsilon$ and $\tilde{K} \ni x_1 \to \mu_{x_1}$ is continuous on $\tilde{K}$. Extend the latter mapping to a continuous mapping $K \ni x_1 \to \nu_{x_1}$ by Tietze’s extension theorem (actually, a generalization thereof to vector valued functions: Dugundji’s theorem, Theorem~4.1 in \cite{dugundji1951extension}). Let $\nu(dx_1,dx_2) = \mu_1(dx_1)\nu_{x_1}(dx_2)$. Then by taking the identity coupling in the first coordinate, we have $\AWD_1(\mu,\nu) \leq T\sqrt{d}\epsilon$, since $\mu(\{\mu_{x_{1}} \neq \nu_{x_1}\}) \leq \epsilon$. Since $\mu$ and $\nu$ are supported on $K$, then by {\color{black}Lemma~\ref{lem:cond_kernel}-(ii)}, for all $\sigma\in (0,1]$, $\mu\ast \calN_{\sigma}$ and $\nu\ast \calN_{\sigma}$ has $\alpha_K$-linear conditional moments, where $\alpha_K = \sup_{x\in K}\Vert x \Vert + M_1(\calN_{1,d})$. Let $C_K = (3+4\alpha_K)^{T}-1$. By \cref{thm:metric_dom_new}, we have for all $\sigma \in (0,1]$,
\begin{equation}\label{eq:lem:bw_limit_compact:1}
\begin{split}
    \AWD_1(\mu\ast\calN_{\sigma},\nu\ast\calN_{\sigma}) &\leq C_K\TVD_1(\mu\ast\calN_{\sigma},\nu\ast\calN_{\sigma})\\
    &= C_K\int_{\R^{dT}} (\Vert x + y \Vert + \frac{1}{2}) \int_{\R^{dT}}\vert \mu - \nu\vert (dy) \calN_\sigma(dx)\\
    &= C_K\int_{\R^{dT}} \int_{\R^{dT}}(\Vert x + y \Vert + \frac{1}{2}) \calN_\sigma(dx) |\mu - \nu|(dy)\\
    &\leq C_K\int_{\R^{dT}} \Big(\int_{\R^{dT}}(\Vert x\Vert + \frac{1}{2})\calN_\sigma(dx) + \Vert y \Vert \Big)|\mu - \nu|(dy) \\
    &\leq C_K(M_1(\calN_1) + \frac{1}{2} + \sup_{y\in K}\Vert y \Vert)\int_{K}|\mu - \nu|(dy) \leq C_K(M_1(\calN_1) + \frac{1}{2} + \sup_{y\in K}\Vert y \Vert)\epsilon.
\end{split}
\end{equation}
Combine \eqref{eq:lem:bw_limit_compact:1}, triangle inequality and \cref{lem:bw_cts_limit_compact} applied to $\nu$. For all $\epsilon > 0$, there exists $C_K, \sigma_{\epsilon} > 0$ s.t. for all $\sigma < \sigma_{\epsilon}$, $\AWD_1(\mu, \mu\ast\calN_{\sigma}) \leq \epsilon + \epsilon + C_K \epsilon$. By re-scaling $\epsilon$, we complete the proof.
\end{proof}
Finally, we relax the compactness assumption in \cref{lem:bw_limit_compact} by approximating any measure in $\calP_1(\R^{dT})$ under $\AWD$-distance by a compactly supported measure; see \cref{lem:compact_approx_2}. 
\begin{theorem}[Stability]\label{thm:bw_limit}
Let $\mu \in \calP_1(\R^{dT})$. Then $\lim_{\sigma \to 0}\AWD_1(\mu,\mu\ast \calN_{\sigma}) = 0$.
\end{theorem}
\begin{proof}[Proof of \cref{thm:bw_limit}]
By \cref{lem:compact_approx_2}, for all $\epsilon>0$, there exists $\nu$ compactly supported s.t. $\AWD_1(\mu,\nu) \leq \epsilon$ and $\AWD_1(\mu\ast\calN_{\sigma},\nu\ast\calN_{\sigma}) \leq \epsilon$. By combining this, triangle inequality and \cref{lem:bw_limit_compact} applied to $\nu$. For all $\epsilon > 0$, there exists $\sigma_{\epsilon} > 0$ s.t. for all $\sigma < \sigma_{\epsilon}$, $\AWD_1(\mu, \mu\ast\calN_{\sigma}) \leq 3\epsilon$. By re-scaling $\epsilon$, we complete the proof.
\end{proof}

\section{Smoothed empirical measures}
\label{sec:smooth_emp}
In this section, we let $\sigma_N$ depend on $N$ and establish the convergence of $\mathsf{S}$-\textit{Emp} to the true underlying measure under $\AWD$-distance. First, we extract $N$ dependency from $C_{\sigma_N,p,M_p}$ in \eqref{eq:C} and from $c_{\sigma_N,K}$ in \eqref{eq:c}.

\begin{lemma}\label{lem:optimal_bandwidth}
    Let $K \subseteq \R^{dT}$ compact. Then there exist $C, c > 0$ independent of $N$ s.t. 
    \begin{equation*}
        C_K C_{\sigma_N,p,M_p}N^{-\frac{1}{2}} = CN^{-r} \quad\text{and}\quad -\frac{N}{c^2_{\sigma_N, K}}\leq -\frac{N^{1-2r}}{c},
    \end{equation*}
    where $C_{\sigma_N,p,M_p}$ is given by \eqref{eq:C} and $c_{\sigma_N,K}$ is given by \eqref{eq:c}.
\end{lemma}
\begin{proof}
By plugging $\sigma_N = N^{-r}$ into $C_{\sigma_N,p,M_p}$, there exists $C >0$ s.t. 
\begin{equation*}
    \begin{split}
        C_K C_{\sigma_N,p,M_p} N^{-\frac{1}{2}} &= C_K \left(\int \frac{(\Vert x \Vert + \frac{1}{2})^2}{1 + \Vert x \Vert^p}dx\right)^{\frac{1}{2}}  \sqrt{ \big(C_1(2^p M_p + 1) + C_2 2^{p} \sigma^{p}\big)\frac{1}{(2\pi\sigma_N)^{dT}}}N^{-\frac{1}{2}}\\
        &= CN^{\frac{rdT}{2}}N^{-\frac{1}{2}} = CN^{-r}.
    \end{split}
\end{equation*}
Similarly, by plugging $\sigma_N = N^{-r}$ into $c_{\sigma_N,K}$, there exists $c_1 >0$ s.t.
\begin{equation*}
    c_{\sigma_N, K} = c_1 (\sup_{x\in K}\frac{1 + 2\Vert x \Vert}{2\sigma_N} + 1) = c_1(\sup_{x\in K}\frac{N^r + 2\Vert x \Vert N^r}{2} + 1) \leq c_1 (\sup_{x\in K}\Vert x\Vert  + 2)N^{r}.
\end{equation*}
Thus, there exists $c > 0$ s.t. $-\frac{N}{c^2_{\sigma_N, K}} \leq -\frac{N^{1-2r}}{c}$.
\end{proof}
\begin{proof}[Proof of Theorem~\ref{thm:smooth_emp}]
    First, we prove the mean convergence rate. Combine \cref{thm:bw_lip_limit}, \cref{thm:smooth_adw_mean_deviation} and triangle inequality. There exists $C_L > 0$ such that, for all $N \in \N$,
    \begin{equation*}
        \begin{aligned}
            \E\Big[\AWD_{1}(\mu , \mu^N \ast \calN_{\sigma_N} ) \Big] &\leq
            \E\Big[\AWD_{1}(\mu , \mu \ast \calN_{\sigma_N} ) \Big] + \E\Big[\AWD^{(\sigma_N)}_{1}(\mu , \mu^N  ) \Big] \leq C_{L}\sigma_N + C_KC_{\sigma_N, p,M_p}N^{-\frac{1}{2}},
        \end{aligned}
    \end{equation*}
    where $C_{\sigma_N,p,M_p}$ is given by \eqref{eq:C}. Deploying \cref{lem:optimal_bandwidth}. there exists $C_0 > 0$ s.t. 
    \begin{equation*}
    \begin{split}
        \E\Big[\AWD_{1}(\mu , \mu^N \ast \calN_{\sigma_N} ) \Big] = C_L N^{-r} +C_0 N^{-r}.
    \end{split}
    \end{equation*}
    By setting $C = C_L + C_0$, we establish \eqref{eq:thm:smooth_emp:0.1}. Next, combining \cref{thm:smooth_adw_mean_deviation} and \cref{lem:optimal_bandwidth}, there exists $C,c_K>0$ s.t. for all $x>0$ and $N \in \N$,
    \begin{equation}\label{eq:thm:smooth_emp_deviation:1}
        \sP\!\left(\AWD_{1}^{(\sigma_N)}(\mu, \mu^N) \geq \frac{x}{2} + CN^{-r} \right) \leq e^{-\frac{x^2N^{1-2r}}{4c^2}}.
    \end{equation}
    By \cref{thm:bw_lip_limit}, there exists $C_L > 0$ s.t. for all $N \in \N$, $\AWD_{1}(\mu , \mu \ast \calN_{\sigma_N} )  \leq C_L \sigma_N = C_L N^{-r}$. By combining this, \eqref{eq:thm:smooth_emp_deviation:1}, and triangle inequality, we have for all $x > 0$ and $N \in \N$, 
    \begin{equation*}
        \begin{split}
            &\quad \sP\Big(\AWD_{1}(\mu , \mu^N \ast \calN_{\sigma_N} ) \geq x + C N^{-r} \Big)\\
            &\leq \sP\Big(\AWD_{1}^{(\sigma_N)}(\mu  , \mu^N) \geq \frac{x}{2} + C N^{-r} \Big) + \sP\Big(\AWD_{1}(\mu , \mu \ast \calN_{\sigma_N} ) \geq \frac{x}{2} +  C N^{-r} \Big)\\
            &= \sP\Big(\AWD_{1}^{(\sigma_N)}(\mu  , \mu^N) \geq \frac{x}{2} + C N^{-r} \Big) \leq e^{-\frac{x^2N^{1-2r}}{4c^2}}.
        \end{split}
    \end{equation*}
    By re-scaling $c$, we establish \eqref{eq:thm:smooth_emp:0.2}. Finally, we prove the almost sure convergence. Notice that \cref{thm:smooth_adw_as} holds for fixed $\sigma$ so we can not simply apply this directly. Nevertheless the proof is similar to the proof of \cref{thm:smooth_adw_as}. By \cref{lem:compact_approx_2}, for all $\epsilon>0$, there exists $\nu$ compactly supported s.t.  
    \begin{equation}\label{eq:thm:smooth_emp_as:1}
        \sup_{\sigma\in(0,1]}\AWD_{1}^{(\sigma)}(\mu, \nu) \leq \epsilon \quad \text{and} \quad \lim_{N\to\infty}\sup_{\sigma\in(0,1]}\AWD_{1}^{(\sigma)}(\mu^N, \nu^N)  \leq \epsilon, \quad \sP\text{-a.s.}
    \end{equation}
    Notice that $\nu$ is compactly supported by construction and $\nu^N$ are empirical measures of $\nu$. By combining \cref{lem:optimal_bandwidth} and \cref{thm:smooth_adw_mean_deviation}, there exists $C, c>0$ s.t. for all $N \in \N$,
    \begin{equation*}
        \sP\!\left(\AWD_{1}^{(\sigma_N)}(\nu, \nu^N ) \geq \frac{x}{2} + CN^{-r} \right) \leq e^{-\frac{x^2N^{1-2r}}{4c^2}}.
    \end{equation*}
    Setting $x = N^{-\frac{r}{2}}$, $\lim_{N \to 0} \frac{1}{2}N^{-\frac{r}{2}} + CN^{-r} = 0 $ and $\sum_{N=1}^{\infty} e^{-\frac{N^{1-3r} }{4c^2}} < \infty$. Thus, by Borel-Cantelli lemma, we have $\lim_{N \to \infty} \AWD_1^{(\sigma_N)}(\nu, \nu^N) = 0,~\sP\text{-a.s.}$ Therefore, by combining this, \eqref{eq:thm:smooth_emp_as:1}, and triangle inequality, we have
    \begin{equation*}
    \begin{split}
        \lim_{N \to \infty} \AWD_1^{(\sigma_N)}(\mu, \mu^N) \leq \sup_{\sigma\in(0,1]}\AWD_{1}^{(\sigma)}(\mu , \nu ) + \lim_{N\to\infty}\sup_{\sigma\in(0,1]}\AWD_{1}^{(\sigma)}(\mu^N , \nu^N) + 0 \leq 2\epsilon, \quad \sP\text{-a.s.}
    \end{split}
    \end{equation*}
    By arbitrarity of $\epsilon$, we get $\lim_{N \to \infty}\AWD_1^{(\sigma_N)}(\mu, \mu^N) = 0$, $\sP$-a.s. Combining this, triangle inequality, and \cref{thm:bw_limit}, we conclude that
\begin{equation*}
\begin{aligned}
    \lim_{N \to \infty}\AWD_{1}(\mu , \mu^N \ast \calN_{\sigma_N} ) \leq \lim_{N \to \infty} \AWD_{1}(\mu, \mu \ast \calN_{\sigma_N} ) +  \lim_{N \to \infty} \AWD_{1}(\mu\ast\calN_{\sigma_N} , \mu^N \ast \calN_{\sigma_N} ) = 0, \quad \sP\text{-a.s.}
\end{aligned}
\end{equation*}
\end{proof}

\section{Smoothness and adapted empirical measures}
\label{sec:smooth_and_ad_emp_meas}
In this section, we establish the convergence of $\mathsf{AS}$-\textit{Emp}. First, we prove the convergence of $\mathsf{AS1}$-\textit{Emp} i.e. the $M=1$ case for $\mathsf{AS}$-\textit{Emp}. Then we establish the convergence of $\mathsf{AS}$-\textit{Emp} for general $M \in \N$.
\subsection{Adapted empirical smoothed measures}
\label{sec:ad_emp_smooth_meas}
Notice that for $\sigma >0$, the $\mathsf{AS1}$-\textit{Emp} of $\mu$ by definition $\hat{\mu}_{\sigma}^{N} \defeq \frac{1}{N}\sum_{n=1}^N \delta_{\hat{\varphi}^{N}(X^{(n)} + \sigma_N\varepsilon^{(n,m)})}$ is the $\mathsf{A}$-\textit{Emp} of $\mu_\sigma = \mu \ast \calN_\sigma$. Therefore, we can prove the convergence of $\mathsf{AS1}$-\textit{Emp} by the convergence theorems of $\mathsf{AS}$-\textit{Emp}; see \cite[Theorem~2.16 and Theorem~2.19]{acciaio2022convergence}.

\begin{theorem}
\label{thm:ad_emp_smooth_mean_deviation}
    Set  $\Delta_N = \sigma_N = N^{-\frac{1}{\calD(d)T}}$ for all $N\in\N$, with $\calD(d) = d$ if $d\geq 3$ and  $\calD(d) = d+1$ if $d=1,2$. Let $L >0$, $\alpha \geq 2$, $\gamma > 0$, $\mu \in \mathcal{P}(\R^{dT})$ with finite $(\alpha,\gamma)$-exponential moment. Assume that $\sup_{x_{1:t}\in\R^{dt}}\calE_{\alpha,2^\alpha \gamma}(\mu_{x_{1:t}}) < \infty$ for all $ t=1,\dots,T-1$ and that for all $\sigma \in (0,1]$, $\mu \ast \calN_{\sigma}$ has $L$-Lipschitz kernels. Then there exist constants $c,C > 0$ s.t., for all $x>0$ and $N\in \N$,
    \begin{equation}          
    \label{eq:thm:ad_emp_smooth_mean_deviation:0.1}
    \E\big[\AWD_1(\mu,\hat{\mu}_{\sigma_N}^{N})\big] \leq CN^{-r},
    \end{equation}
    \begin{equation}
    \label{eq:thm:ad_emp_smooth_mean_deviation:0.2}
        \sP\Big(\AWD_1(\mu,\hat{\mu}_{\sigma_N}^{N}) \geq  x + CN^{-r} \Big) \leq Ce^{-c Nx^2},
    \end{equation}
    and $\lim_{N \to \infty}\AWD_1(\mu,\hat{\mu}_{\sigma_N}^{N}) = 0$, $\sP$-a.s.  
\end{theorem}
\begin{proof}
    For $\sigma \in (0,1]$ we first check that $\mu_\sigma$ satisfies the exponential moments assumption of \cite[Theorem~2.16 and Theorem~2.19]{acciaio2022convergence}. {\color{black} By \cref{lem:cond_kernel}-(iii), }$\mu_\sigma$ has uniform $(\alpha,\gamma)$-exponential moment kernels for all $\sigma\in (0,1]$. On the other hand, by assumptions, for all $\sigma \in (0,1]$, $\mu_\sigma$ has $L$-Lipschitz kernels. Therefore, we can apply Theorem~2.16 (i) (with $p > \frac{d}{d-1}$) and Theorem~2.19 (i) in \cite{acciaio2022convergence} to $\mu_\sigma$ with $N$ many samples, for all $\sigma \in (0,1]$. Then there exist constants $c,C > 0$ such that, for all $\sigma \in (0,1]$, $x \geq 0$ and $N \in \N$,
    \begin{equation}
    \label{eq:thm:ad_emp_smooth_mean_deviation:1}
    \E\Big[\AWD_1(\mu_{\sigma},\hat{\mu}_{\sigma}^{N})\Big] \leq CN^{-\frac{1}{\calD(d)T}}, \quad \sP\Big(\AWD_1(\mu_{\sigma},\hat{\mu}_{\sigma}^{N}) \geq  x + CN^{-\frac{1}{\calD(d)T}} \Big) \leq Ce^{-cNx^2}.
    \end{equation}
    On the other hand, by \cref{thm:bw_lip_limit}, there exists $C_L >0$ s.t. for all $N \in \N$,
    \begin{equation}
    \label{eq:thm:ad_smooth_emp_mean_deviation:2}
    \begin{split}
        \AWD_{1}(\mu , \mu_{\sigma_N} ) \leq C_L \sigma_N = C_L N^{-\frac{1}{\calD(d)T}}.
    \end{split} 
    \end{equation}
    Therefore, by combining \eqref{eq:thm:ad_smooth_emp_mean_deviation:1}, \eqref{eq:thm:ad_smooth_emp_mean_deviation:2} and triangle inequality, we prove \eqref{eq:thm:ad_emp_smooth_mean_deviation:0.1} and \eqref{eq:thm:ad_emp_smooth_mean_deviation:0.2}. Furthermore, by combining this and Borel-Cantelli as in the proof of \cref{lem:smooth_adw_as_compact}, we prove almost sure convergence.
\end{proof}

\subsection{Adapted smoothed empirical measures}
\label{sec:ad_smooth_emp_meas}

In this subsection, we establish the convergence of adapted smoothed empirical measures. Recall the definition of adapted smoothed empirical measures that 
\begin{equation*}
    {\color{black}\tilde{\mu}_{\sigma,\zeta}^{N,M}} \defeq \frac{1}{M}\sum_{m=1}^{M}\tilde{\mu}^N_{\zeta,m},\quad \tilde{\mu}^N_{\zeta,m} = (x\mapsto x+\zeta^m)_{\#}\tilde{\mu}^N_m,\quad \tilde{\mu}^N_m\defeq \frac{1}{N}\sum_{n=1}^N \delta_{\hat{\varphi}^{N}(X^{(n)} + \sigma\varepsilon^{(n,m)})},
\end{equation*}
where $\hat{\varphi}^{N}$ is the adapted projection in \cref{def:uniform_adapted_empirical_measure}, $(\zeta^m)_{m=1}^{M}$ are $M$ distinct points in $(0,\frac{1}{2G_N})^{dT}$. Without $(\zeta^m)_{m=1}^{M}$, all $(\tilde{\mu}^M_m)_{m=1}^{M}$ are supported on the same grid $\hat{\Lambda}^N$. Then some measures might have intersection on the support. Since the adapted Wasserstein distance is so sensitive to the support that it is not convex with respect to its marginal; see \cref{ex:awd_nonconvex} for a counterexample. However with {\color{black}distinct points} $(\zeta^m)_{m=1}^{M}$ introduced, $(\hat{\Lambda}^N + \zeta^m)_{m=1}^{M}$ are distinct grids such that $(\tilde{\mu}^M_{\zeta,m})_{m=1}^{M}$ has no intersection in the support. This allows us to decouple bicausal couplings on distinct supports to establish convexity of the adapted Wasserstein distance. Also, by choosing $(\zeta^m)_{m=1}^{M}$ from $(0,\frac{1}{2G_N})^{dT}$, the shifting error {\color{black} $ \AWD_1(\tilde{\mu}^N_{\zeta,m}, \tilde{\mu}^N_{m}) \leq \frac{\sqrt{d}T}{2}\frac{1}{G_N} \leq \frac{\sqrt{d}T}{2}N^{-r}$} is absorbed by $\calO(N^{-r})$.
\begin{lemma}
\label{lem:awd_convex}
    Let $M \in \N$, $\mu \in \calP(\R^{dT})$, $\nu = \frac{1}{M}\sum_{m=1}^{M}\nu^m$, and $\nu^m \in \calP(\R^{dT})$  for all $m=1,\dots,M$. Assume $(\nu^m)_{m=1}^M$ have distinct supports {\color{black}i.e. for all $i,j = 1,\dots, M$ and $i\neq j$, $\operatorname{supp}(\nu^i) \cap \operatorname{supp}(\nu^j) = \emptyset$}. Then $\AWD_1(\mu, \nu) \leq \frac{1}{M}\sum_{m=1}^{M}\AWD_1(\mu,\nu^m)$.
\end{lemma}
\begin{proof}
Let $\pi^{m}\in\bccpl(\mu,\nu^m)$, $m=1,\dots,M$ and define $\pi = \frac{1}{M}\sum_{m=1}^{M}\pi^{m}$. First, we notice that $\pi \in \cpl(\mu,\nu)$ since marginals are interchangeable with average. Thus, we only remain to prove that $\pi$ is a bi-causal coupling. We prove it by inspecting whether $\pi_{x_{1:t},y_{1:t}} \in \cpl(\mu_{x_{1:t}}, \nu_{y_{1:t}})$ for $\pi$-a.s. $(x_{1:t},y_{1:t}) \in \R^{2dt}$ and $t=1,\dots,T-1$. Notice that $(\nu^m)_{m=1}^M$ have distinct supports, which we denote by $(\Lambda_m)_{m=1}^{M}$, {\color{black}$\Lambda_m = \operatorname{supp}(\nu^m)$, $m = 1,\dots, M$}. We have $\nu^m(\Lambda_{m'}) =  \mathbbm{1}_{\{m=m'\}}$ and $\pi^m(\R^{dT}\times\Lambda_{m'}) = \mathbbm{1}_{\{m=m'\}}$. Thus, for $\nu$-a.s. $y_{1:t} \in \R^{dt}$ and $\pi$-a.s. $(x_{1:t},y_{1:t}) \in \R^{2dt}$, $\frac{1}{M}\frac{d\nu^m}{d\nu}(y_{1:t})=\mathbbm{1}_{\{y_{1:t} \in \Lambda_m\}}$, $\frac{1}{M}\frac{d\pi^m}{d\pi}(x_{1:t},y_{1:t})=\mathbbm{1}_{\{y_{1:t} \in \Lambda_m\}}$. Therefore,
\begin{equation}
\label{eq:lem:awd_convex:1}
\begin{split}
    \pi_{x_{1:t},y_{1:t}} &= \frac{1}{M}\sum_{m=1}^{M}\frac{d\pi^m}{d\pi}(x_{1:t},y_{1:t})\pi^m_{x_{1:t},y_{1:t}} = \sum_{m=1}^{M}\mathbbm{1}_{\{y_{1:t} \in \Lambda_m\}}\pi^m_{x_{1:t},y_{1:t}},\\
    \nu_{y_{1:t}}(dy_{t+1}) &= \frac{1}{M}\sum_{m=1}^{M}\frac{d\nu^m}{d\nu}(y_{1:t})\nu^m_{y_{1:t}}(dy_{t+1}) = \sum_{m=1}^{M}\mathbbm{1}_{\{y_{1:t} \in \Lambda_m\}}\nu^m_{y_{1:t}}(dy_{t+1}).
\end{split}
\end{equation}
Since $\pi^{m}\in\bccpl(\mu,\nu^m)$, we have $\pi^m_{x_{1:t},y_{1:t}} \in \cpl(\mu_{x_{1:t}}, \nu^m_{y_{1:t}})$. Combining this and \eqref{eq:lem:awd_convex:1}, we have
\begin{equation*}
\left\{\begin{aligned}
    \pi_{x_{1:t},y_{1:t}}(dx_{t+1}) &= \sum_{m=1}^{M}\mathbbm{1}_{\{y_{1:t} \in \Lambda_m\}}\pi^m_{x_{1:t},y_{1:t}}(dx_{t+1}) = \mu_{x_{1:t}}(dx_{t+1}),\\
    \pi_{x_{1:t},y_{1:t}}(dy_{t+1}) &= \sum_{m=1}^{M}\mathbbm{1}_{\{y_{1:t} \in \Lambda_m\}}\pi^m_{x_{1:t},y_{1:t}}(dy_{t+1}) = \sum_{m=1}^{M}\mathbbm{1}_{\{y_{1:t} \in \Lambda_m\}}\nu^m_{y_{1:t}}(dy_{t+1}) = \nu_{y_{1:t}}(dy_{t+1}),
\end{aligned}\right.
\end{equation*}
which proves that $\pi \in \bccpl(\mu,\nu)$. Therefore, we conclude that 
\begin{equation*}
    \AWD_1(\mu,\nu) \leq \frac{1}{M}\sum_{m=1}^{M}\inf_{\pi^m\in\bccpl(\mu,\nu^m)}\int \Vert x-y\Vert \pi^{m}(dx,dy) = \frac{1}{M}\sum_{m=1}^{M}\AWD_1(\mu,\nu^m).
\end{equation*}
\end{proof}

\begin{example}
\label{ex:awd_nonconvex}
Let $\epsilon>0$, $M=2$, $\mu,\nu^1,\nu^2,\nu \in\calP(\R^{2})$ s.t. $\mu = \nu^1 = \frac{1}{2}\delta_{(\epsilon,1)} +  \frac{1}{2}\delta_{(-\epsilon,-1)}$, $\nu^2 = \frac{1}{2}\delta_{(-\epsilon,1)} +  \frac{1}{2}\delta_{(\epsilon,-1)}$, and $\nu = \frac{1}{2}\nu^1 + \frac{1}{2}\nu^2$; see Figure~\ref{fig:ex} for visualization. Then we have $\AWD_1(\mu,\nu^1) = 0$, $\AWD_1(\mu,\nu^2) = 2\epsilon$, and $\AWD_1(\mu,\nu) = 1$. By choosing $\epsilon<1$, we get $\AWD_1(\mu,\nu) = 1  > \epsilon = \frac{1}{M}\sum_{m=1}^{M}\AWD_1(\mu,\nu^m)$.
\begin{figure}[H]
    \centering
    \begin{subfigure}{.3\textwidth}
      \centering
      \includegraphics[width=0.8\linewidth]{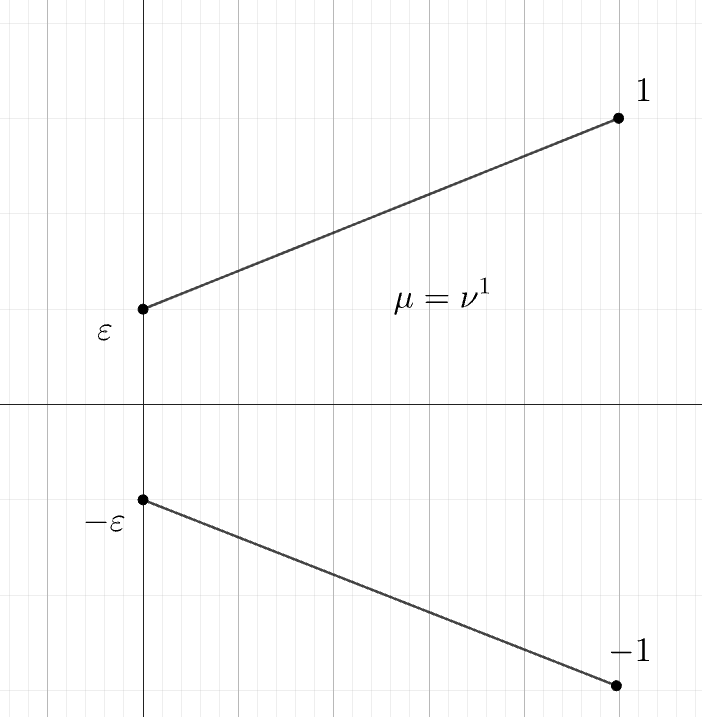}
      \caption{$\mu = \nu^1$}
    \end{subfigure}
    \begin{subfigure}{.3\textwidth}
      \centering
      \includegraphics[width=0.8\linewidth]{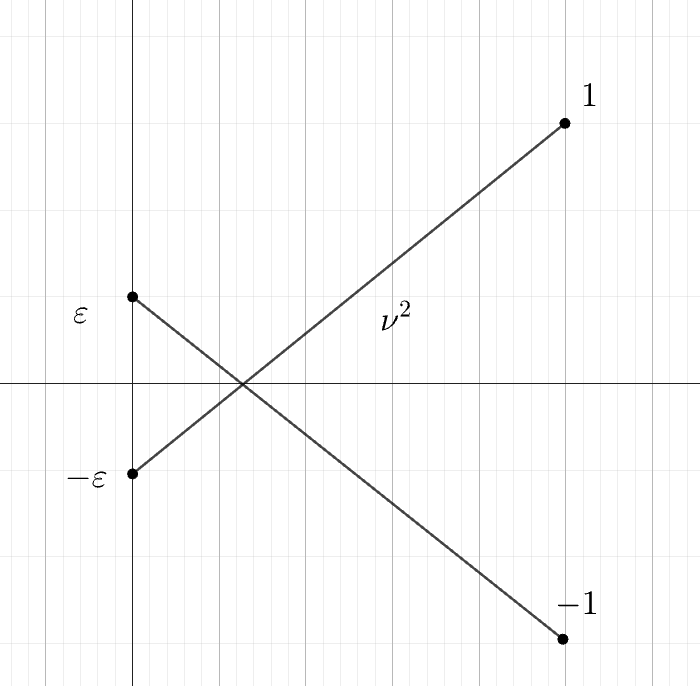}
      \caption{$\nu^2$}
    \end{subfigure}
    \begin{subfigure}{.3\textwidth}
      \centering
      \includegraphics[width=0.8\linewidth]{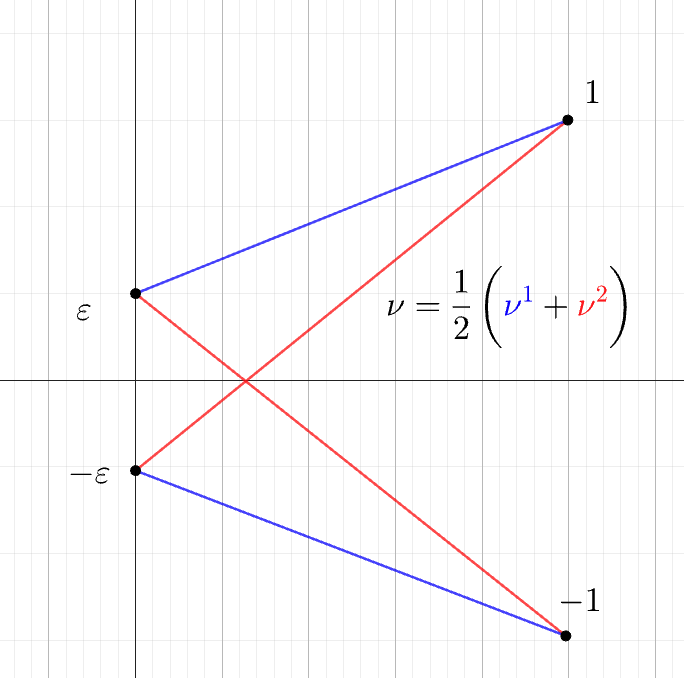}
      \caption{$\nu$}
    \end{subfigure}
    \caption{Visualization of $\mu, \nu^1, \nu^2$ and $\nu$.}
    \label{fig:ex}
\end{figure}
\end{example}

\begin{proof}[Proof of \cref{thm:ad_smooth_emp_mean_deviation}]
    Recall the definition of adapted smoothed empirical measures that 
    \begin{equation*}
        {\color{black}\tilde{\mu}_{\sigma_N,\zeta}^{N,M}} \defeq \frac{1}{M}\sum_{m=1}^{M}\tilde{\mu}^N_{\zeta,m},\quad \tilde{\mu}^N_{\zeta,m} = (x\mapsto x+\zeta^m)_{\#}\tilde{\mu}^N_m,\quad \tilde{\mu}^N_m\defeq \frac{1}{N}\sum_{n=1}^N \delta_{\hat{\varphi}^{N}(X^{(n)} + \sigma_N\varepsilon^{(n,m)})},
    \end{equation*}
    where $\hat{\varphi}^{N}$ is the adapted projection in \cref{def:uniform_adapted_empirical_measure}, $(\zeta^m)_{m=1}^{M}$ are $M$ distict points in $(0,\frac{1}{2G_N})^{dT}$. Notice that $(\tilde{\mu}^N_m)_{m=1}^{M}$ have distinct supports. By \cref{lem:awd_convex}, we have
    \begin{equation}\label{eq:thm:ad_smooth_emp_mean_deviation:1}
        \AWD_1(\mu, {\color{black}\tilde{\mu}_{\sigma_N,\zeta}^{N,M}}) \leq  \frac{1}{M}\sum_{m=1}^M\AWD_1(\mu, \tilde{\mu}^N_{\zeta,m}) \leq \frac{1}{M}\sum_{m=1}^M\AWD_1(\mu, \tilde{\mu}^N_{m}) + \frac{\sqrt{d}T}{2}\frac{1}{G_N}.
    \end{equation}
    Let $C_0 = \frac{\sqrt{d}T}{2}$. Therefore, by combining \eqref{eq:thm:ad_smooth_emp_mean_deviation:1}, \cref{thm:ad_emp_smooth_mean_deviation} and the fact that $\frac{1}{G_N} \leq \Delta_N = N^{-r}$, we prove \eqref{eq:thm:ad_smooth_emp_mean_deviation:0}. Furthermore, by combining this and Borel-Cantelli as in the proof of \cref{lem:smooth_adw_as_compact}, we prove almost sure convergence. This completes the proof of \cref{thm:ad_smooth_emp_mean_deviation}.
\end{proof}

We end this section proving a popular class of unbounded measures, which satisfy the assumption of \cref{thm:ad_smooth_emp_mean_deviation}.
\begin{example}[Gaussian mixture model]
    \label{ex:models}
    Let $K \in \N$ and $\mu \in \calP(\R^{dT})$ with density $p_{\mu}(x) = \sum_{k=1}^{K}w_k\varphi_{\sigma^k}(x-x^k)$, where $\sum_{k=1}^{K} w_k = 1$, $w_k\geq 0, \sigma^k \geq 0$, $x^k \in \R^{dT}$ for all $k=1,\dots,K$. For all $\sigma \in (0,1]$, $\mu\ast \calN_\sigma$ has density 
    \begin{equation*}
    \begin{split}
        p_{\mu \ast \calN_\sigma}(x) &= \int \varphi_\sigma(x-y)p_\mu(y)dy = \sum_{k=1}^{K}w_k\int\varphi_\sigma(x-y)\varphi_{\sigma^k}(y-x^k)dy\\
        &= \sum_{k=1}^{K}w_k\int\varphi_\sigma(x-y)\varphi_{\sigma^k}(y-x^k)dy = \sum_{k=1}^{K}w_k\varphi_{\sqrt{(\sigma^k)^2 + \sigma^2}}(x-x^k).
    \end{split}
    \end{equation*}
    Therefore, for all $x_{1:t}\in \R^{dt}, t=1,\dots,T-1$, the kernel $(\mu\ast\calN_\sigma)_{x_{1:t}}$ has density 
    \begin{align*}
        p_{\mu \ast \calN_\sigma}(x_{t+1}| x_{1:t}) = \sum_{k=1}^{K} \alpha_k(x_{1:t})\varphi_{\sqrt{(\sigma^k)^2 + \sigma^2}}(x_{t+1}-x_{t+1}^k),
    \end{align*}
    where $\alpha_k(x_{1:t}) = \frac{w_k\varphi_{\sqrt{(\sigma^k)^2 + \sigma^2}}(x_{1:t}-x^k_{1:t})}{\sum_{k=1}^{K}w_k\varphi_{\sqrt{(\sigma^k)^2 + \sigma^2}}(x_{1:t}-x^k_{1:t})}$. Thus for all $x_{1:t}, x'_{1:t} \in \R^{dt}$, $t=1,\dots,T-1$,
    \begin{align*}
        &\quad~\WD_1((\mu\ast\calN_\sigma)_{x_{1:t}}, (\mu\ast\calN_\sigma)_{x'_{1:t}})\\
        &\leq \sup_{k,k'=1,\dots,K}\WD_1\Big(\calN\big(x_{1:t}^k, (\sigma^k)^2 + \sigma^2\big), \calN\big(x_{1:t}^{k'}, (\sigma^{k'})^2 + \sigma^2\big)\Big)\sum_{k=1}^{K}\vert \alpha_k(x_{1:t}) - \alpha_k(x'_{1:t})\vert .
    \end{align*}
    Since $\alpha_k$ is Lipschitz in $x_{1:t}$, there exists $L>0$ s.t. for all $\sigma \in (0,1]$, $\mu \ast \calN_{\sigma}$ has $L$-Lipschitz kernels. Moreover, notice that the Gaussian mixture model has Gaussian tail in both density and conditional density. Thus, $\mu$ has finite $(2,1)$-exponential moment and $\sup_{x_{1:t}\in\R^{dt}}\calE_{2,4}(\mu_{x_{1:t}}) < \infty$. Therefore, $\mu$ satisfies the assumption in \cref{thm:ad_smooth_emp_mean_deviation}.
\end{example}

\appendix

\section{Appendix}
\label{sec:appendix}

\color{black}
\subsection{Conditional kernel}
\label{sec:cond_kernel}
\begin{lemma}
    \label{lem:cond_kernel}
    Let $\mu \in \calP_1(\R^{dT})$, $\sigma \in (0,1]$, and we denote by $\mu_\sigma \coloneqq \mu \ast \calN_\sigma$. We have:
    \begin{enumerate}[label = (\roman*)]
        \item For all $t = 1,\dots,T-1$, $x_{1:t} \in \R^{dt}$,
        \begin{equation}
        \label{eq:lem:cond_kernel.1}
            (\mu\ast\calN_\sigma)_{x_{1:t}} = \int_{\R^{dt}}(\mu_{y_{1:t}}\ast\calN_{\sigma,d}) \frac{\varphi_{\sigma,dt}(x_{1:t} - y_{1:t})\mu_{1:t}(dy_{1:t})}{\int_{\R^{dt}} \varphi_{\sigma,dt}(x_{1:t} - y'_{1:t})\mu_{1:t}(dy'_{1:t})}.
        \end{equation}
        \item Let $K$ be a compact subset of $\R^{dT}$ and assume $\mu \in \calP(K)$. Then for all $t = 1,\dots,T-1$, $x_{1:t} \in \R^{dt}$,
        \begin{equation}
            \label{eq:lem:cond_kernel.2}
            M_1\big((\mu_\sigma)_{x_{1:t}}\big) \leq \sup_{x\in K}\Vert x \Vert + M_1(\calN_{1,d}).
        \end{equation}
        \item Let $\alpha \geq 2$, $\gamma > 0$, and assume $\sup_{x_{1:t}\in\R^{dt}}\calE_{\alpha,2^\alpha \gamma}(\mu_{x_{1:t}}) < \infty$ for all $ t=1,\dots,T-1$. Then for all $t = 1,\dots,T-1$, $x_{1:t} \in \R^{dt}$,    
        \begin{equation}
            \label{eq:lem:cond_kernel.3}
            \mathcal{E}_{\alpha,\gamma}\big((\mu_\sigma)_{x_{1:t}}\big) \leq \calE_{\alpha,2^\alpha \gamma}(\calN_{1,d})\sup_{y_{1:t}\in\R^{dt}}\calE_{\alpha,2^\alpha \gamma}(\mu_{y_{1:t}}) < \infty.
        \end{equation}
        \item Let $L>0$ and assume $\mu$ has $L$-Lipschitz kernels. Then for all $t = 1,\dots,T-1$, $x_{1:t} \in \R^{dt}$,
        \begin{equation}
            \label{eq:lem:cond_kernel.4}
            \int_{\R^{dt}}\WD_{1}(\mu_{x_{1:t}},(\mu_\sigma)_{x_{1:t}})\mu_\sigma(dx_{1:t}) \leq \sigma \big(M_1(\calN_{1,d}) + L M_1(\calN_{1,dt})\big).
        \end{equation}

        \item Assume $\mu$ has uniformly continuous kernels. Then for all $\epsilon > 0$ there exists $\sigma_\epsilon > 0$ s.t. for all $\sigma \in (0,\sigma_\epsilon)$, $t=1,\dots,T-1$, $x_{1:t} \in \R^{dt}$, 
        \begin{equation}
        \label{eq:lem:cond_kernel.5}
            \int_{\R^{dt}} \WD_{1}(\mu_{x_{1:t}},(\mu_\sigma)_{x_{1:t}})\mu_\sigma(dx_{1:t}) \leq \epsilon.
        \end{equation}            
\end{enumerate}
\end{lemma}
\begin{proof}
    As a convoluted measure, $\mu_\sigma$ has continuous density, which we denote by
    \begin{equation*}
        p_{\mu_\sigma}(x) \coloneqq \int_{\R^{dT}}\varphi_\sigma(x-y)\mu(dy).
    \end{equation*}
    By Bayes' rule, we have for all $t=1,\dots,T-1$, $x_{1:t} \in \R^{dt}$,
    \begin{equation*}
    \begin{split}
        p_{\mu_\sigma}(x_{t+1}|x_{1:t}) = \frac{p_{\mu_\sigma}(x_{1:t+1})}{p_{\mu_\sigma}(x_{1:t})} &= \frac{\int_{\R^{dt}} \int_{\R^d} \varphi_{\sigma,d}(x_{t+1} - y_{t+1})\mu_{y_{1:t}}(dy_{t+1}) \varphi_{\sigma,dt}(x_{1:t} - y_{1:t})\mu_{1:t}(dy_{1:t})}{\int_{\R^{dt}} \varphi_{\sigma,dt}(x_{1:t} - y'_{1:t})\mu_{1:t}(dy'_{1:t})}
     \end{split}
    \end{equation*}
    Notice that $p_{\mu_{y_{1:t}}\ast \calN_{\sigma,d}}(x_{t+1})\coloneqq\int_{\R^d} \varphi_{\sigma,d}(x_{t+1} - y_{t+1})\mu_{y_{1:t}}(dy_{t+1})$ is the density function of $\mu_{y_{1:t}}\ast \calN_{\sigma,d}$. Thus, we get 
    \begin{equation*}
    \begin{split}
        p_{\mu_\sigma}(x_{t+1}|x_{1:t}) = \frac{p_{\mu_\sigma}(x_{1:t+1})}{p_{\mu_\sigma}(x_{1:t})}
        &= \int_{\R^{dt}} p_{\mu_{y_{1:t}}\ast \calN_{\sigma,d}}(x_{t+1})\frac{ \varphi_{\sigma,dt}(x_{1:t} - y_{1:t})\mu_{1:t}(dy_{1:t})}{\int_{\R^{dt}} \varphi_{\sigma,dt}(x_{1:t} - y'_{1:t})\mu_{1:t}(dy'_{1:t})},
    \end{split}
    \end{equation*}
    which proves that \eqref{eq:lem:cond_kernel.1}. For notational simplicity, in the proof below, we let 
    \begin{equation*}
        w_{\sigma,t}(x_{1:t},dy_{1:t}) \coloneqq \frac{\varphi_{\sigma,dt}(x_{1:t} - y_{1:t})\mu_{1:t}(dy_{1:t})}{\int_{\R^{dt}} \varphi_{\sigma,dt}(x_{1:t} - y'_{1:t})\mu_{1:t}(dy'_{1:t})},
    \end{equation*}
    and write
    \begin{equation*}
        (\mu\ast\calN_\sigma)_{x_{1:t}} = \int_{\R^{dt}}(\mu_{y_{1:t}}\ast\calN_{\sigma,d}) w_{\sigma,t}(x_{1:t},dy_{1:t}).
    \end{equation*}
    Next, we prove (ii). For all $\sigma \in (0,1]$, $t=1,\dots,T-1$ and $x_{1:t} \in \R^{dt}$,
    \begin{equation*}
    \begin{split}
    \int_{\R^d} \Vert x_{t+1} \Vert d(\mu\ast\calN_{\sigma}) _{x_{1:t}}
    &= \int_{\R^{d}}\int_{\R^{dt}} \Vert x_{t+1}\Vert (\mu_{y_{1:t}}\ast\calN_{\sigma,d})(dx_{t+1}) w_{\sigma, dt}(x_{1:t},dy_{1:t})  \\
    &\leq \int_{\R^{d}}\Big(\int_{\R^{dt}} \Vert x_{t+1}\Vert \mu_{y_{1:t}}(dx_{t+1}) + \int_{\R^{dt}} \Vert x_{t+1}\Vert\calN_{\sigma,d}(dx_{t+1})\Big)w_{\sigma, dt}(x_{1:t},dy_{1:t})\\
    &= \int_{\R^{d}}\int_{\R^{dt}} \Vert x_{t+1}\Vert \mu_{y_{1:t}}(dx_{t+1})w_{\sigma, dt}(x_{1:t},dy_{1:t}) + M_1(\calN_{\sigma,d})\\
    &\leq \sup_{x\in K}\Vert x \Vert + M_1(\calN_{1,d}),
    \end{split}
    \end{equation*}
    which proves \eqref{eq:lem:cond_kernel.2}. Next, we prove (iii). For all $\sigma \in (0,1]$, $t=1,\dots,T-1$ and $x_{1:t} \in \R^{dt}$,
    \begin{equation*}
    \begin{split}
    \mathcal{E}_{\alpha,\gamma}((\mu\ast\calN_{\sigma})_{x_{1:t}}) &= \int_{\R^d} \exp\left(\gamma\Vert x_{t+1}\Vert^{\alpha}\right)\int_{\R^{dt}}(\mu_{y_{1:t}}\ast\calN_{\sigma,d})(dx_{t+1}) w_{\sigma,t}(x_{1:t},dy_{1:t}) \\
    &= \int_{\R^{dt}}w_{\sigma,t}(x_{1:t},dy_{1:t})\int_{\R^d} \exp\left(\gamma\Vert x_{t+1}\Vert^{\alpha}\right)(\mu_{y_{1:t}}\ast\calN_{\sigma,d})(dx_{t+1})  \\
    &= \int_{\R^{dt}}w_{\sigma,t}(x_{1:t},dy_{1:t})\E_{Z,\eta}\left[\exp\left(\gamma\Vert Z_{y_{1:t}} + \sigma \eta\Vert^{\alpha}\right)\right]\quad \text{$\big((Z_{y_{1:t}},\eta) \sim \mu_{y_{1:t}}\otimes\calN_{1,d}\big)$}\\
    &\leq \int_{\R^{dt}}w_{\sigma,t}(x_{1:t},dy_{1:t})\E_{Z_{y_{1:t}},\eta}\left[\exp\left(2^{\alpha}\gamma\Vert Z_{y_{1:t}}\Vert^{\alpha}\right)\cdot \exp\left(2^{\alpha} \gamma\Vert\sigma\eta\Vert^{\alpha}\right)\right]\\
    &\leq \calE_{\alpha,2^\alpha \gamma}(\calN_{1,d}) \int_{\R^{dt}}w_{\sigma,t}(x_{1:t},dy_{1:t})\E_{Z_{y_{1:t}}}\left[\exp\left(2^{\alpha}\gamma\Vert Z_{y_{1:t}}\Vert^{\alpha}\right)\right] \\
    &\leq \calE_{\alpha,2^\alpha \gamma}(\calN_{1,d})\sup_{y_{1:t}\in\R^{dt}}\calE_{\alpha,2^\alpha \gamma}(\mu_{y_{1:t}}) < \infty,
    \end{split}
    \end{equation*}
    which proves \eqref{eq:lem:cond_kernel.3}. Next, we prove (iv). By assumption, $\mu$ has $L$-Lipschitz kernels. Thus, we obtain that
    \begin{equation}\label{eq:thm:bw_lip_limit:3}
        \begin{aligned}
        \WD_1(\mu_{x_{1:t}},(\mu_\sigma)_{x_{1:t}})&= \WD_1\big(\mu_{x_{1:t}},\int (\mu_{y_{1:t}}\ast \calN_{\sigma,d}) w_{\sigma,t}(x_{1:t},dy_{1:t})\big)\\
        &\leq \int_{\R^{dt}} \WD_1\big(\mu_{x_{1:t}},\mu_{y_{1:t}}\ast \calN_{\sigma,d}\big) w_{\sigma,t}(x_{1:t},dy_{1:t})\\
        &\leq \int_{\R^{dt}} \left(\WD_1\big(\mu_{x_{1:t}},\mu_{y_{1:t}}\big) + \WD_1\big(\mu_{y_{1:t}},\mu_{y_{1:t}}\ast \calN_{\sigma,d}\big)\right) w_{\sigma,t}(x_{1:t},dy_{1:t})\\
        &\leq \sigma M_1(\calN_{1,d}) + L \int_{\R^{dt}}   \Vert x_{1:t} - y_{1:t}\Vert w_{\sigma,t}(x_{1:t},dy_{1:t}).
        \end{aligned}
    \end{equation}
    Notice that 
    \begin{equation}\label{eq:thm:bw_lip_limit:4}
        \begin{aligned}
            &\quad~ \int_{\R^{dt}}  \int_{\R^{dt}}  \Vert x_{1:t} - y_{1:t}\Vert w_{\sigma,t}(x_{1:t},dy_{1:t}) \mu_\sigma(dx_{1:t})\\
            &= \int_{\R^{dt}} \int_{\R^{dt}}  \Vert x_{1:t} - y_{1:t}\Vert\frac{\varphi_{\sigma,dt}(x_{1:t} - y_{1:t})\mu_{1:t}(dy_{1:t})}{\int_{\R^{dt}} \varphi_{\sigma,dt}(x_{1:t} - y'_{1:t})\mu_{1:t}(dy'_{1:t})} \mu_\sigma(dx_{1:t}) \\ 
            &= \int_{\R^{dt}} \int_{\R^{dt}}  \Vert x_{1:t} - y_{1:t}\Vert\frac{\varphi_{\sigma,dt}(x_{1:t} - y_{1:t})\mu(dy_{1:t})}{p_{\mu_\sigma}(x_{1:t})} p_{\mu_\sigma}(x_{1:t}) dx_{1:t}\\
            &= \int_{\R^{dt}} \Big(\int_{\R^{dt}}  \Vert x_{1:t} - y_{1:t}\Vert \varphi_{\sigma,dt}(x_{1:t} - y_{1:t})  dx_{1:t}\Big)\mu(dy_{1:t})
            = \sigma M_1(\calN_{1,dt}).
        \end{aligned}
    \end{equation}
    By combining \eqref{eq:thm:bw_lip_limit:3} and \eqref{eq:thm:bw_lip_limit:4}, we obtain that for all $t=1,\dots,T-1$,
    \begin{equation*}
        \int_{\R^{dt}} \WD_{1}(\mu_{x_{1:t}},(\mu_\sigma)_{x_{1:t}})\mu_\sigma(dx_{1:t}) \leq \sigma \Big(M_1(\calN_{1,d}) + L M_1(\calN_{1,dt})\Big),
    \end{equation*}  
    which proves \eqref{eq:lem:cond_kernel.4}. Finally, we prove (v). By assumption, $\mu$ has uniformly continuous kernels, hence almost Lipschitz continuous. Therefore, for all $\epsilon > 0$ there exists $L_\epsilon > 0$ s.t. for all $t=1,\dots,T-1$, $x_{1:t}, y_{1:t} \in \R^{dt}$, 
    \[
    \WD_1\big(\mu_{x_{1:t}},\mu_{y_{1:t}}\big) \leq L_\epsilon \Vert x_{1:t} - y_{1:t}\Vert + \epsilon . 
    \]
    Therefore, following a similar proof in \eqref{eq:thm:bw_lip_limit:3} and \eqref{eq:thm:bw_lip_limit:4}, we get 
    \begin{equation*}
        \int_{\R^{dt}} \WD_{1}(\mu_{x_{1:t}},(\mu_\sigma)_{x_{1:t}})\mu_\sigma(dx_{1:t}) \leq \sigma \big(M_1(\calN_{1,d}) + L_\epsilon M_1(\calN_{1,dt})\big) + \epsilon.
    \end{equation*}  
    By choosing $\sigma_\epsilon = \inf_{t=1,\dots,T-1}\big(M_1(\calN_{1,d}) + L_\epsilon M_1(\calN_{1,dt})\big)^{-1}\epsilon$ and rescaling $\epsilon$, we prove \eqref{eq:lem:cond_kernel.5}.
\end{proof}
\color{black}

\subsection{Compact approximation}
\label{sec:app:compact_approx}

\begin{lemma}\label{lem:compact_approx_1}
    For all $R\geq 1$, there exists $\phi_R \colon \R^{dT} \to [-R-1,R+1]^{dT}$ s.t. for all $x \in \R^{dT}$, $\Vert \phi_{R}(x) \Vert \leq 2T\sqrt{d}\Vert x \Vert$ and for all $\mu \in \calP(\R^{dT})$, 
    \begin{equation*}
        \AWD_1(\mu, {\phi_R}_{\#}\mu) \leq (1+2T\sqrt{d})\int_{K_R^c} \Vert x\Vert \mu(dx).
    \end{equation*}
\end{lemma}
\begin{proof}
    We define the compact cubes for all $t=1,\dots,T$ by $K_{R,1:t} = [-R,R]^{dt}$, $K_{R, 1} = K_{R,1:1}$, $K_{R} = K_{R, 1:T}$ and let $x^{R} = (R+1, \dots, R+1)\in K_R^c$. For $x_1 \in \R^{d}$ and $x\in \R^{dT}$, we define $\phi_{R,1} \colon \R^{d} \to \R^{d}$ and $\phi_{R}\colon \R^{dT} \to \R^{dT}$ by 
    \begin{equation*}
        \phi_{R,1}(x_1) = \left\{\begin{aligned}
            x_1, \quad &x_1 \in K_{R,1} \\
            x^R_1, \quad &x_{1}\not\in K_{R,1} \\
        \end{aligned}\right. ,\quad
        \phi_{R}(x) = (x_{1:\tau-1},x_{\tau:T}^{R}), \quad \tau=\inf\{t\colon x_t \notin K_{R,1}\} .
    \end{equation*}
    With the projections defined above, we are ready to construct a coupling $\pi$ with the first marginal $\mu$ and second marginal compactly supported, denoted by $\nu$. We define the coupling $\pi$ iteratively by 
    \begin{equation*}
        \pi(dx_{1:T},dy_{1:T}) =\pi_1(dx_1,dy_1)\prod_{t=1}^{T-1}\pi_{x_{1:t},y_{1:t}}(dx_{t+1},dy_{t+1}),
    \end{equation*}
    where $\pi_1 = (\mathbf{id}, \phi_{R,1})_{\#}\mu_1$ and for all $t = 1,\dots,T-1$, $x_{1:t}, y_{1:t} \in \R^{dt}$, 
    \begin{equation*}
     \pi_{x_{1:t},y_{1:t}} = \left\{\begin{aligned}
    (\mathbf{id}, \phi_{R,1})_{\#}\mu_{x_{1:t}}, \quad &x_{1:t} \in K_{R,1:t}, y_{1:t} = x_{1:t} \\
    \mu_{x_{1:t}} \otimes \delta_{x_{t+1}^{R}}, \quad &\text{otherwise} \\
    \end{aligned}\right. .
    \end{equation*} 
    Intuitively, we couple identically until the path goes beyond the compact cube. We claim that $\pi_{1:t}(K_{R,1:t}\times K_{R,1:t} \cap \{x_{1:t}\neq y_{1:t}\}) = 0$ for all $t=1,\dots,T$. First, we notice that $\pi_1 = (\mathbf{id}, \phi_{R,1})_{\#}\mu_1 = (\mathbf{id}, \mathbf{id})_{\#}\mu_1\vert_{K_{R,1}} + \mu_1\vert_{K^{c}_{R,1}}\otimes \delta_{x^{R}_1}$, where $\mu_1\vert_{K_{R,1}}(dx_1) = \mu_1(dx_1\cap K_{R,1})$ and $\mu_1\vert_{K_{R,1}^{c}}(dx_1) = \mu_1(dx_1\cap K_{R,1}^{c})$. Thus $\pi_1(K_{R,1}\times K_{R,1} \cap \{x_1\neq y_1\}) = 0$. Then by induction, assuming $\pi_{1:t}(K_{R,1:t}\times K_{R,1:t} \cap \{x_{1:t}\neq y_{1:t}\}) = 0$, we have 
    \begin{equation}
    \label{eq:lem:compact_approx:0.1}
    \begin{split}
        &\quad~ \pi_{1:t+1}(K_{R,1:t+1}\times K_{R,1:t+1} \cap \{x_{1:t+1}\neq y_{1:t+1}\})\\
        &\leq \pi_{1:t+1}(K_{R,1:t+1}\times K_{R,1:t+1} \cap \{x_{1:t}\neq y_{1:t}\}) + \pi_{1:t+1}(K_{R,1:t+1}\times K_{R,1:t+1} \cap \{x_{1:t} = y_{1:t}, x_{t+1}\neq y_{t+1}\})\\
        &\leq \pi_{1:t}(K_{R,1:t}\times K_{R,1:t} \cap \{x_{1:t}\neq y_{1:t}\}) + \pi_{1:t+1}(K_{R,1:t+1}\times K_{R,1:t+1} \cap \{x_{1:t} = y_{1:t}, x_{t+1}\neq y_{t+1}\})\\
        &= \pi_{1:t+1}(K_{R,1:t+1}\times K_{R,1:t+1} \cap \{x_{1:t} = y_{1:t}, x_{t+1}\neq y_{t+1}\})\\
        &= \int_{K_{R,1:t}\times K_{R,1:t} \cap \{x_{1:t} = y_{1:t}\}} \int_{K_{R,1}\times K_{R,1} \cap \{x_{t+1}\neq y_{t+1}\}} \pi_{x_{1:t},y_{1:t}}(dx_{t+1},dy_{t+1}) \pi_{1:t}(dx_{1:t},dy_{1:t}).
    \end{split}
    \end{equation}
    Then by the definition of $\pi_{x_{1:t},y_{1:t}}$, we have for all $(x_{1:t},y_{1:t})\in K_{R,1:t}\times K_{R,1:t} \cap \{x_{1:t} = y_{1:t}\}$,
    \begin{equation}
    \label{eq:lem:compact_approx:0.2}
        \begin{split}
        \int_{K_{R,1}\times K_{R,1} \cap \{x_{t+1}\neq y_{t+1}\}} \pi_{x_{1:t},y_{1:t}}(dx_{t+1},dy_{t+1}) &=  (\mathbf{id}, \phi_{R,1})_{\#}\mu_{x_{1:t}}(K_{R,1}\times K_{R,1} \cap \{x_{t+1}\neq y_{t+1}\}) = 0.
        \end{split}
    \end{equation}
    Combining \eqref{eq:lem:compact_approx:0.1}, \eqref{eq:lem:compact_approx:0.2} and the induction, we complete the proof of the claim. Now we are ready to check that $\pi\in\bccpl(\mu,\nu)$. On one hand, by definition, $\pi_{x_{1:t},y_{1:t}}(dx_{t+1}) = \mu_{x_{1:t}}(dx_{t+1})$. On the other hand, we know from the claim above that $x_{1:t} = y_{1:t}$ $\pi_{1:t}$-a.s. on $K_{R,1:t}\times K_{R,1:t}$. Thus 
    \begin{equation*}
        \pi_{x_{1:t},y_{1:t}}(dy_{t+1}) = \left\{\begin{aligned}
            {\phi_{R,1}}_{\#}\mu_{y_{1:t}}, \quad & y_{1:t} \in K_{R,1:t} \\
            \delta_{x_{t+1}^{R}}, \quad &\text{otherwise}\\
        \end{aligned}\right.  \quad \pi_{1:t}\text{-a.s.}
    \end{equation*}
    Therefore for all $t=1,\dots,T-1$, $\pi_{x_{1:t},y_{1:t}} \in \cpl(\mu_{x_{1:t}},\nu_{y_{1:t}})$, which proves that $\pi\in\bccpl(\mu,\nu)$. Moreover, it is easy to check that $\nu = {\phi_{R}}_{\#}\mu$ by construction. By construction of $\phi_{R}$, $\nu$ is compactly supported and for all $x\in \R^{dT}$, 
    \begin{equation*}
    \begin{split}
        \Vert \phi_R(x) \Vert &= \Vert (x_{1:\tau-1},x_{\tau:T}^{R}) \Vert = \sum_{t=1}^{\tau-1}\Vert x_t\Vert + \sum_{t=\tau}^{T}\Vert x_t^{R}\Vert = \sum_{t=1}^{\tau-1}\Vert x_t\Vert + \sum_{t=\tau}^{T} \sqrt{d}(R + 1)\\
        &\leq \sum_{t=1}^{\tau-1}\Vert x_t\Vert + \sum_{t=\tau}^{T} \sqrt{d}(R + 1)\frac{\Vert x_\tau \Vert }{R} \leq \sum_{t=1}^{\tau-1}\Vert x_t\Vert + \sum_{t=\tau}^{T} 2\sqrt{d}\Vert x_\tau \Vert \leq 2T\sqrt{d}\Vert x \Vert.
    \end{split}
    \end{equation*}
    Since we have already defined a bi-causal coupling between $\mu$ and $\nu$, that is $\pi = (\mathbf{id}, \phi)_{\#}\mu \in \bccpl(\mu,\nu)$, by the definition of adapted Wasserstein distance we have 
    \begin{equation*}
        \AWD_1(\mu,{\phi_R}_{\#}\mu) \leq \int_{\R^{dT}} \Vert x-\phi_R(x)\Vert \mu(dx) = \int_{K_R^c} \Vert x-\phi_R(x)\Vert \mu(dx) \leq (1+2T\sqrt{d})\int_{K_R^c} \Vert x\Vert \mu(dx).
    \end{equation*}
\end{proof}

\begin{lemma}
\label{lem:compact_approx_2}
    Let $\mu \in \calP_1(\R^{dT})$. Then for all $\epsilon >0$ there exists $\nu\in \calP(\R^{dT})$ compactly supported such that
    \begin{enumerate}[label=(\roman*)]
        \item $\sup_{\sigma\in [0,1]}\AWD_1^{(\sigma)}(\mu,\nu) \leq \epsilon$,
        \item $\lim_{N\to\infty}\sup_{\sigma\in [0,1]}\AWD_1^{(\sigma)}(\mu^N,\nu^N) \leq \epsilon$, $\sP$-a.s.,
    \end{enumerate}
    where $\mu^N$ and $\nu^N$ are empirical measures of $\mu$ and $\nu$.
\end{lemma}
\begin{proof}[Proof of Lemma~\ref{lem:compact_approx_2}]
    Let $R_1 \geq R_2 \geq 1$, $K_{R_1} = [-R_1,R_1]^{dT}, K_{R_2} = [-R_2,R_2]^{dT}$, $\phi_{R_1}, \phi_{R_2}$ defined as in \cref{lem:compact_approx_1} and $\nu = {\phi_{R_1}}_{\#}\mu$. By triangle inequality, for all $\sigma \in [0,1]$, 
    \begin{equation}
    \label{eq:lem:compact_approx_2:1}
    \begin{split}
    \AWD_1(\mu\ast\calN_\sigma,\nu\ast\calN_\sigma) &\leq \AWD_1({\phi_{R_2}}_{\#}(\mu\ast\calN_\sigma),{\phi_{R_2}}_{\#}(\nu\ast\calN_\sigma))  \\
        & +\AWD_1(\mu\ast\calN_\sigma,{\phi_{R_2}}_{\#}(\mu\ast\calN_\sigma))  + \AWD_1(\nu\ast\calN_\sigma,{\phi_{R_2}}_{\#}(\nu\ast\calN_\sigma)).
    \end{split}
    \end{equation}
    For the last two terms, by \cref{lem:compact_approx_1} we have
    \begin{equation}
    \label{eq:lem:compact_approx_2:2}
    \begin{split}
        \AWD_1(\mu\ast\calN_\sigma,{\phi_{R_2}}_{\#}(\mu\ast\calN_\sigma)) \leq (1+2T\sqrt{d})\int_{K^{c}_{R_2}}\Vert x \Vert d(\mu\ast\calN_\sigma),\\
        \AWD_1(\nu\ast\calN_\sigma,{\phi_{R_2}}_{\#}(\nu\ast\calN_\sigma)) \leq (1+2T\sqrt{d})\int_{K^{c}_{R_2}}\Vert x \Vert d(\nu\ast\calN_\sigma).
        \end{split}
    \end{equation}
    For the first term, by \cref{thm:metric_dom_new} (with $\alpha =\sqrt{d}TR_2$), we have
    \begin{equation}
    \label{eq:lem:compact_approx_2:3}
        \begin{split}
            \AWD_1({\phi_{R_2}}_{\#}(\mu\ast\calN_\sigma),{\phi_{R_2}}_{\#}(\nu\ast\calN_\sigma)) &\leq ((3+4\sqrt{d}TR_2)^{T} - 1) \TVD_1({\phi_{R_2}}_{\#}(\mu\ast\calN_\sigma),{\phi_{R_2}}_{\#}(\nu\ast\calN_\sigma)).
        \end{split}
    \end{equation}
    Moreover, notice that 
   \begin{equation}
   \label{eq:lem:compact_approx_2:4}
    \begin{split}
         &\quad~ \TVD_1({\phi_{R_2}}_{\#}(\mu\ast\calN_\sigma),{\phi_{R_2}}_{\#}(\nu\ast\calN_\sigma)) \leq\TVD_1(\mu\ast\calN_{\sigma}, \nu\ast\calN_{\sigma})\\ 
         &= \int_{\R^{dT}} (2\Vert x + y \Vert + 1) \int_{\R^{dT}}\vert \mu - \nu\vert (dy) \calN_\sigma(dx)\\
        &=\int_{\R^{dT}} \int_{\R^{dT}}(2\Vert x + y \Vert + 1) \calN_\sigma(dx) |\mu - \nu|(dy)\\
        &\leq\int_{\R^{dT}} \Big(\int_{\R^{dT}}(2\Vert x\Vert + 1)\calN_\sigma(dx) + 2\Vert y \Vert \Big)|\mu - \nu|(dy) \\
        &= (2M_1(\calN_\sigma)+1)\int_{\R^{dT}}|\mu - \nu|(dy) + 2\int_{\R^{dT}} \Vert y \Vert |\mu - \nu|(dy)\\
        &\leq (2M_1(\calN_1)+1)\int_{\R^{dT}}|\mu - \nu|(dy) + 2\int_{\R^{dT}} \Vert y \Vert |\mu - \nu|(dy) \\
        &\leq (2M_1(\calN_1)+1)\big(\mu(K_{R_1}^c) + \nu(K_{R_1}^c)\big) + 2\big(\int_{K_{R_1}^c} \Vert x \Vert \mu(dx) + \int_{K_{R_1}^c} \Vert y \Vert \nu(dy)\big)\\
        &\leq (2M_1(\calN_1)+3)\big(\int_{K_{R_1}^c} \Vert x \Vert \mu(dx) + \int_{K_{R_1}^c} \Vert y \Vert \nu(dy)\big) \leq (2M_1(\calN_1)+3)(1+2T\sqrt{d})\int_{K_{R_1}^c} \Vert x \Vert \mu(dx),
    \end{split}
    \end{equation}
    where the second to last inequality is because $R_1 \geq 1$. Therefore, by combining \eqref{eq:lem:compact_approx_2:1}, \eqref{eq:lem:compact_approx_2:2}, \eqref{eq:lem:compact_approx_2:3} and \eqref{eq:lem:compact_approx_2:4}, we have for all $\sigma \in [0,1]$, 
    \begin{equation}
    \label{eq:lem:compact_approx_2:5}
    \begin{split}
    \AWD_1^{(\sigma)}(\mu,\nu) &\leq (1+2T\sqrt{d})\int_{K^{c}_{R_2}}\Vert x \Vert d(\mu\ast\calN_\sigma) +(1+2T\sqrt{d})\int_{K^{c}_{R_2}}\Vert x \Vert d(\nu\ast\calN_\sigma)\\
    & + ((3+4\sqrt{d}TR_2)^{T} - 1)(2M_1(\calN_1)+3)(1+2T\sqrt{d})\int_{K_{R_1}^c} \Vert x \Vert \mu(dx).
    \end{split}
    \end{equation}
    For the uniform integrability of the first two terms in \eqref{eq:lem:compact_approx_2:5}, notice that for all $\sigma \in [0,1]$,
    \begin{equation}
    \label{eq:lem:compact_approx_2:6}
    \begin{split}
        \int_{K^{c}_{R_2}}\Vert x \Vert d(\mu\ast\calN_\sigma) &= \E_{X\sim \mu, \varepsilon\sim \calN_1}[\Vert X+\sigma \varepsilon\Vert \mathbbm{1}_{\{X+\sigma \varepsilon \notin K_{R_2} \}}]\\
        &\leq \E_{X\sim \mu, \varepsilon\sim \calN_1}[\Vert X+\sigma \varepsilon\Vert \mathbbm{1}_{\{X\notin K_{\frac{R_2}{2}} \}}] + \E_{X\sim \mu, \varepsilon\sim \calN_1}[\Vert X+\sigma \varepsilon\Vert \mathbbm{1}_{\{\sigma \varepsilon \notin K_{\frac{R_2}{2}} \}}]\\
        &\leq \E_{X\sim \mu}[\Vert X\Vert \mathbbm{1}_{\{X\notin K_{\frac{R_2}{2}} \}}] + \E_{\varepsilon\sim\calN_1}[\Vert \sigma \varepsilon\Vert]\,\E_{X\sim \mu}[ \mathbbm{1}_{\{X\notin K_{\frac{R_2}{2}} \}}]\\
        &+ \E_{X\sim \mu}[\Vert X\Vert ]\,\E_{\varepsilon\sim \calN_1}[\mathbbm{1}_{\{\sigma \varepsilon \notin K_{\frac{R_2}{2}} \}}] + \E_{\varepsilon\sim \calN_1}[\Vert \sigma \varepsilon\Vert \mathbbm{1}_{\{\sigma \varepsilon \notin K_{\frac{R_2}{2}} \}}]\\
        &\leq  \int_{K^{c}_{\frac{R_2}{2}}}\Vert x \Vert d\mu + M_1(\calN_1)\int_{K^{c}_{\frac{R_2}{2}}} d\mu + M_1(\mu)\int_{K^{c}_{\frac{R_2}{2}}} d\calN_1 + \int_{K^{c}_{\frac{R_2}{2}}}\Vert x \Vert d\calN_1.
    \end{split}        
    \end{equation}
    where $K^{c}_{\frac{R_2}{2}} = [-\frac{R_2}{2}, \frac{R_2}{2}]^{dT}$. Similar for $\nu$, we have for all $\sigma \in [0,1]$,
    \begin{equation}
    \label{eq:lem:compact_approx_2:7}
    \begin{split}
        &\quad~\int_{K^{c}_{R_2}}\Vert x \Vert d(\nu\ast\calN_\sigma)\\
        &\leq \int_{K^{c}_{\frac{R_2}{2}}}\Vert x \Vert d\nu + M_1(\calN_1)\int_{K^{c}_{\frac{R_2}{2}}} d\nu + M_1(\nu)\int_{K^{c}_{\frac{R_2}{2}}} d\calN_1 + \int_{K^{c}_{\frac{R_2}{2}}}\Vert x \Vert d\calN_1\\
        &\leq 2T\sqrt{d}\int_{K^{c}_{\frac{R_2}{2}}}\Vert x \Vert d\mu + M_1(\calN_1)\int_{K^{c}_{\frac{R_2}{2}}} d\mu + 2T\sqrt{d}M_1(\mu)\int_{K^{c}_{\frac{R_2}{2}}} d\calN_1 + \int_{K^{c}_{\frac{R_2}{2}}}\Vert x \Vert d\calN_1,
    \end{split}
    \end{equation}
    where the last inequality is because $\Vert \phi_{R_1}(x) \Vert \leq 2T\sqrt{d}\Vert x \Vert$ for all $x \in \R^{dT}$ and $R_1 \geq R_2$. Therefore, by \eqref{eq:lem:compact_approx_2:6}, \eqref{eq:lem:compact_approx_2:7} and the integrability of $\mu$, there exists $R_2$ large enough such that  
    \begin{equation}
    \label{eq:lem:compact_approx_2:8}
        (1+2T\sqrt{d})\int_{K^{c}_{R_2}}\Vert x \Vert d(\mu\ast\calN_\sigma) \leq \epsilon\quad \text{and}\quad (1+2T\sqrt{d})\int_{K^{c}_{R_2}}\Vert x \Vert d(\nu\ast\calN_\sigma)\leq \epsilon,
    \end{equation}
    and $R_1$ large enough such that 
    \begin{equation}
    \label{eq:lem:compact_approx_2:9}
        ((3+4\sqrt{d}TR_2)^{T} - 1)(2M_1(\calN_1)+3)(1+2T\sqrt{d})\int_{K_{R_1}^c} \Vert x \Vert \mu(dx) \leq \epsilon.
    \end{equation}
    Plugging \eqref{eq:lem:compact_approx_2:8} and \eqref{eq:lem:compact_approx_2:9} into \eqref{eq:lem:compact_approx_2:5}, we obtain that $\AWD_1(\mu\ast\calN_\sigma,\nu\ast\calN_\sigma) \leq 3 \epsilon$ and by re-scaling $\epsilon$, we complete the proof of (i). For (ii), by the law of large number, we have
    \begin{equation*}
        \lim_{N\to\infty}\int_{K_{R_1}^c} \Vert x \Vert \mu^N(dx) = \int_{K_{R_1}^c} \Vert x \Vert \mu(dx),\quad  \lim_{N\to\infty}\int_{K^{c}_{\frac{R_2}{2}}} \Vert x \Vert \mu^N(dx) = \int_{K^{c}_{\frac{R_2}{2}}} \Vert x \Vert \mu(dx), 
    \end{equation*}
    and $\lim_{N\to\infty}\int_{K^{c}_{\frac{R_2}{2}}}\mu^N(dx) = \int_{K^{c}_{\frac{R_2}{2}}}\mu(dx).$ Thus by replacing $\mu, \nu$ by $\mu^N, \nu^N$ in the proof of (i) and applying the law of large number, we can similarly prove (ii).
\end{proof}

\printbibliography
\end{document}